\documentclass[a4paper, 12pt]{amsart}
\usepackage{amsmath, amsthm, amscd, amssymb, amsfonts, amsxtra, amssymb, latexsym}
\usepackage{enumerate}
\usepackage{verbatim}
\usepackage{float}
\usepackage{caption} 
\usepackage{array}
\usepackage[hyperindex]{hyperref}
\hypersetup{colorlinks = true,	allcolors  = blue}

\newcommand{\Z}{\mathbb{Z}}
\newcommand{\N}{\mathbb{N}}
\newcommand{\ff}{\mathbb{F}}

\newcommand{\G}{\Gamma}

\newcommand{\sk}{\smallskip}
\newcommand{\msk}{\medskip}

\newtheorem{thm}{Theorem}[section]
\newtheorem{prop}[thm]{Proposition}
\newtheorem{lem}[thm]{Lemma}
\newtheorem{coro}[thm]{Corollary}
\theoremstyle{definition}
\newtheorem{rem}[thm]{Remark}
\newtheorem{exam}[thm]{Example}
\newtheorem{defi}[thm]{Definition}
\newtheorem{quest}[thm]{Question}

\theoremstyle{remark}

\usepackage{color}
    
\def\blue{\color{blue}}

\def\black{\color{black}}

\begin{document}
\numberwithin{equation}{section}
\title[On regular graphs equienergetic with their complements]{On regular graphs equienergetic \\ with their complements}
\author[R.A.\@ Podest\'a, D.E.\@ Videla]{Ricardo A.\@ Podest\'a, Denis E.\@ Videla}
\dedicatory{\today}
\thanks{Partially supported by CONICET and SECyT-UNC}

\address{Ricardo A.\@ Podest\'a, FaMAF -- CIEM (CONICET), Universidad Nacional de C\'ordoba, Av.\@ Medina Allende 2144, Ciudad Universitaria, (5000) C\'ordoba, Rep\'ublica Argentina. \newline
{\it E-mail: podesta@famaf.unc.edu.ar}}

\address{Denis E.\@ Videla, FaMAF -- CIEM (CONICET), Universidad Nacional de C\'ordoba, Av.\@ Medina Allende 2144, Ciudad Universitaria, (5000) C\'ordoba, Rep\'ublica Argentina.  \newline{\it E-mail: devidela@famaf.unc.edu.ar}}

\begin{abstract}
We give necessary and sufficient conditions on the parameters of a regular graph $\G$ (with or without loops) such that $E(\G)=E(\overline \G)$.  
We study complementary equienergetic cubic graphs obtaining classifications up to isomorphisms for connected cubic graphs with single loops (5 non-isospectral pairs) and connected integral cubic graphs without loops ($\G=K_3\square K_2$ or $Q_3$).
Then we show that, up to complements, the only bipartite regular graphs equienergetic and non-isospectral with their complements 
are the crown graphs $Cr(n)$ or $C_4$. 
Next, for the family of strongly regular graphs $\G$ we characterize all possible parameters $srg(n,k,e,d)$ such that 
$E(\G) = E(\overline \G)$. Furthermore, using this, we prove that a strongly regular graph is equienergetic to its complement if and only if it is either a conference graph or else it is a pseudo Latin square graph (i.e.\@ has $OA$ parameters).
We also characterize all complementary equienergetic pairs of graphs of type $\mathcal{C}(2)$, $\mathcal{C}(3)$ and $\mathcal{C}(5)$ in Cameron's hierarchy (the cases $\mathcal{C}(1)$ and $\mathcal{C}(4)$ are still open). Finally, we consider unitary Cayley graphs over rings $G_R=X(R,R^*)$. 
We show that if $R$ is a finite Artinian ring with an even number of local factors, 
then $G_R$ is complementary equienergetic if and only if 
$R=\ff_{q} \times \ff_{q'}$ is the product of 2 finite fields.  
\end{abstract}

\maketitle

\section{Introduction}
Let $\Gamma$ be a graph of $n$ vertices. The eigenvalues of $\Gamma$ are the eigenvalues $\{\lambda_i\}_{i=1}^n$ of its adjacency matrix. The \textit{spectrum} of $\Gamma$, denoted 
$$Spec(\Gamma) = \{ [\lambda_{i_1}]^{e_{i_1}}, \ldots,[\lambda_{i_s}]^{e_{i_s}}\},$$ 
is the set of all the different eigenvalues $\{\lambda_{i_j}\}$ of $\Gamma$ counted with their multiplicities $\{e_{i_j}\}$  (sometimes we will omit the braces). 
The spectrum of $\G$ is called \textit{symmetric} if the multiplicities of $\lambda$ and $-\lambda$ coincide for any $\lambda$, that is $m(\lambda) = m(-\lambda)$ for every $\lambda \in Spec(\G)$, and \textit{integral} if $Spec(\G) \subset \Z$. Sometimes one simply says that $\G$ is symmetric or integral instead of saying that the spectrum of $\G$ is symmetric or integral, respectively. 
The \textit{energy }of $\Gamma$ is defined by 
$$E(\Gamma) = \sum_{i=1}^n |\lambda_i|.$$
We refer to the books \cite{BH} or \cite{CDS} for a complete viewpoint of spectral theory of graphs, and to \cite{Gu} for a survey on energy of graphs, see also the book \cite{LSG} which contains many open problems related to energy of graphs.

Let $\Gamma_1$ and $\Gamma_2$ be two graphs with the same number of vertices. The graphs are said \textit{isospectral} if $Spec(\Gamma_1) = Spec(\Gamma_2)$ and \textit{equienergetic} if $E(\Gamma_1) = E(\Gamma_2)$.
It is clear by the definitions that isospectrality implies equienergeticity, but the converse does not hold in general.
There are many papers on these problems (see for instance \cite{Ba}, \cite{GPI}, \cite{HX}, \cite{Il}, \cite{RW}, \cite{Ra}, \cite{Ra2} and the references therein).

This work deals with the following question:
\textit{which regular graphs are equienergetic (and non-isospectral) with their own complements?} 
If a graph $\G$ and its complement $\overline \G$ are equienergetic we will say, as in \cite{A+} or \cite{RPPAG}, that they are \textit{complementary equienergetic} graphs.
Self-complementary graphs are trivially complementary equienergetic, so the interest is put on non self-complementary graphs. 

In the literature, there are few examples or classifications of complementary equienergetic graphs. There is a classification for double line graphs: if $\Gamma$ is a regular graph, then $L^2(\Gamma)$ and $\overline{L^2(\Gamma)}$ are equienergetic if and only if $\Gamma=K_6$ \cite{RGWH}. Ali et al \cite{A+} recently determined all (possible) complementary equienergetic graphs $\G$ and all (possible) complementary equienergetic line graphs $L(\G)$, where $\G$ has at most 10 vertices.
They also showed that the incidence graph $IG(\ell,\ell-1,\ell-2)$ of a symmetric 2-design $2-(\ell,\ell-1,\ell-2)$ is complementary equienergetic \cite{A+}. In \cite{RPPA}, Ramane et al give some pairs of complementary equienergetic graphs, such as the line graphs of complete bipartite graphs $L(K_{m,n})$ with $m,n \ge 2$ and two families of strongly regular graphs having parameters $srg(4n^2,2n^2-n, n^2-n, n^2-n)$ with $n>1$ and $srg(n^2,3n-3, n,6)$ with $n>2$. These families of strongly regular graphs are part of a general family (see Remark~\ref{rem cite})).
Furthermore, Ramane et al \cite{RPPAG} recently proved that strongly regular graphs having orthogonal array ($OA$) parameters are complementary equienergetic.

In this paper we deal with regular graphs (possibly with multiple loops at the vertices). By a systematic approach, we will give a complete answer to the question of complementary equienergeticity for bipartite regular graphs, for strongly regular graphs 
and for certain unitary Cayley graphs over rings, without loops in all the cases.

\subsubsection*{Outline and results}
We now give the structure  of the paper and summarize its main results.
In Section~2, we study complementary equienergetic regular graphs (with or without loops) in general.
In Proposition \ref{lem delta} we give an equivalent condition for the equienergeticity between $\G$ and $\overline{\G}$
in terms of the parameters of $\G$ and other invariant that we define in \eqref{discrepancy}, distinguishing three cases: no loops, single loops, and multiple loops per vertex. In fact, if $\G$ is $k$ regular with $n$ vertices and $m$ is the maximum number of loops per vertex, then $\G$ is complementary equienergetic if and only if 
$$n=2k+f(m)$$ 
where $f(m)$ is certain integer number depending whether $m=0$, $m=1$ or $m\ge 2$. We give some examples of complementary equienergetic graphs with $m= 1$ or $nm=2$.

As an application, in Section 3 we classify all connected cubic graphs with single loops which are complementary equienergetic (Proposition \ref{cubic with loops}, only 5 non-isospectral pairs) and all connected integral cubic graphs equienergetic with their complements (Proposition \ref{integral cubic}, $\G$ must be $K_3\square K_2$ or the cube $Q_3$). Also, we show that there are no distance-regular cubic graphs complementary equienergetic (Corollary \ref{DRG cubic}) and no arc-transitive cubic graphs with girth $g<6$ equienergetic with their complements (Corollary \ref{arc trans g6}).

In the next section, we consider bipartite regular graphs.
We show that, up to complements, the only bipartite regular graphs equienergetic with their own complements 
are the crown graphs $Cr(n)$ and the 4-cycle $C_4$. In both cases, $\G$ and $\overline{\G}$ are non-isospectral.

In the following 3 sections, \S 5 through \S 7, we study the family of strongly regular graphs $srg(n,k,e,d)$.
In Section 5 we begin by showing that the only imprimitive connected strongly regular graph is the complete multipartite graph 
$K_{m\times m}$ with $m\ge 2$. 
Then, in \eqref{srg equien}, we give a condition for complementary equienergeticity of primitive strongly regular graphs in terms of the parameters $n,k,e$ and $d$.
One of the main results in the paper is Theorem~\ref{Teo equien srg e le d}, which provides a classification of the parameters 
of complementary equienergetic primitive strongly regular graphs. There are three possibilities, either $\G$ is a conference graph with parameters $srg(4d+1,2d,d-1,d)$ for $d\ge 1$ or else it has parameters
$$srg\big( (2\ell + \epsilon)^2, (\ell-h+\epsilon)(2\ell-1+\epsilon), d+2h,d \big)$$
where $d=(\ell-h)(\ell-h+1)$ with $\epsilon \in \{0,1\}$ for some integers $\ell, h$.

In Section 6, as an application of the results in Section 5, we consider many subfamilies of strongly regular graphs and 
classify all complementary equienergetic graphs within these families.
Namely, we characterize complementary equienergetic strongly regular having integral minimum eigenvalue $s=-m$ with $m\ge 2$ (Propositions \ref{s=-2} and \ref{s=-m}), triangle-free strongly regular graphs (Proposition \ref{Tf srgs}), semiprimitive generalized Paley graphs (Proposition \ref{eqnoisocomp}) and strongly regular graphs which are uniquely determined by their spectrum (Proposition \ref{DS srg}).

In Section 7, we give a full characterization of complementary equienergetic strongly regular graphs.
The block graph of an orthogonal array is a strongly regular graph with parameters 
$$srg \big(n^2, m(n-1), m^2-3m+n, m(m-1) \big).$$ 
Any strongly regular graph having these parameters is said to be a pseudo Latin square graph or that it has $OA(n,m)$ parameters. 
A simple calculation shows that strongly regular graphs with $OA$ parameters are equienergetic with their complement (see Proposition~\ref{OA equien}). This was first proved in \cite{RPPAG}. 
The main result in the paper is the fact that the converse also holds, that is, if a strongly regular graph is complementary equienergetic then it must have $OA$ parameters. More precisely,    
by using Theorem \ref{Teo equien srg e le d}, we show that all primitive strongly regular graphs equienergetic with their complement are either conference graphs or else have $OA$ parameters. 
Finally, in Theorem \ref{Equien noisosp prim. car} we show that all strongly regular graphs equienergetic 
and non-isospectral with their complements have $OA$ parameters.
As a consequence of results in previous sections and Proposition \ref{teo Ct}, we characterize all complementary equienergetic pairs of graphs of type $\mathcal{C}(1)$ (regular graphs) in the bipartite case, $\mathcal{C}(2)$ (strongly regular graphs), $\mathcal{C}(3)$ and $\mathcal{C}(5)$ in Cameron's hierarchy (the cases $\mathcal{C}(1)$ for non-bipartite graphs and $\mathcal{C}(4)$ are still open).

In the last section, we study another family of regular graphs which are not strongly regular in general, the unitary Cayley graphs over rings. Let $G_R$ be the Cayley graph $X(R,R^*)$ where $R$ is a finite commutative ring with identity. 
Such a ring has Artin decomposition $R=R_1 \times \cdots \times R_s$ where each $R_i$ is local. We show that if $R$ has an even number of local factors ($s$ even), then $G_R$ and $\overline{G}_R$ are complementary equienergetic if and only if $R$ is the product of two finite fields, i.e.\@ $R=\ff_{q_1}\times \ff_{q_2}$. In this case, the graph $G_R$ is an strongly regular graph. The classification of complementary equienergetic unitary Cayley graphs where $R$ has an odd number $s \ge 3$ of local factors seems difficult and remains open.

\section{Equienergy conditions for $\G$ and $\overline \G$}
In this section, we obtain a simple condition for the equienergeticity of $\G$ and $\overline \G$, when $\G$ is a regular graph  (possibly with loops), that will be used throughout the paper. 

Consider the real functions
\begin{equation} \label{delta}
\delta(x):=|1+x|-|x| = \begin{cases}
 1 		& \quad \text{if $x\ge 0$}, \\[1mm]
2x+1 	& \quad \text{if $-1\le x\le 0$}, \\[1mm] 
-1 		& \quad \text{if $x\le-1$}.
\end{cases}
\end{equation}
and, for each $m\ge 1$, 
\begin{equation} \label{deltam}
\delta_m(x):=|x-m|-|x| = \begin{cases}
-m		& \quad \text{if $x\ge m$}, \\[1mm]
-2x+m 	& \quad \text{if $0\le x \le m$}, \\[1mm] 
m 		& \quad \text{if $x\le 0$}.
\end{cases}
\end{equation}

We will need the following related invariants of a graph. 
If $\G$ is a $k$-regular graph with $n$ vertices having eigenvalues $k=\lambda_1 \ge \lambda_1 \ge \cdots \ge \lambda_{n}$, we define the numbers
\begin{equation} \label{discrepancy}
\Delta(\G) := \sum_{\lambda\in Sp'(\Gamma)} \delta(\lambda) \qquad \quad  \text{and} \qquad  \quad 
\Delta_m(\G) := \sum_{\lambda\in Sp'(\Gamma)} \delta_m(\lambda)
\end{equation}
for $m\ge 1$, where 
$$Sp'(\Gamma) = Spec(\G) \smallsetminus \{\lambda_1\} = \{\lambda_2, \ldots, \lambda_n\}.$$ 
Note that $\delta(x)$ coincides with the sign function $sgn(x)$ in $I^c$, where $I=(-1,0]$. Thus, by \eqref{delta} and \eqref{discrepancy} we have
$$ \Delta(\G) = \sum_{\lambda \in Sp'(\Gamma) \cap I^c} sgn(\lambda) + \sum_{\lambda \in Sp'(\Gamma) \cap I} (2\lambda +1). $$
We can write this more appropriately in the form 
\begin{equation} \label{sigma}
\Delta(\G)  = \sum_{\lambda \in Sp'(\Gamma), \, |\lambda|\ge 1} sgn(\lambda) + T + m(0) + S ,
\end{equation}
where $m(0)$ is the multiplicity of the 0-eigenvalue, 
\begin{equation} \label{TS}
T = \# \{ \lambda \in Sp'(\Gamma) : \lambda \in (0,1) \} \qquad \text{and} \qquad 
S = \sum_{\lambda \in Sp'(\Gamma) \cap (-1,0)} (2\lambda +1).
\end{equation} 
A similar expression as in \eqref{sigma} can be given for $\Delta_m(\G)$ for any $m\ge 2$.

\begin{defi}
We call the number $\Delta(\G)$ defined in \eqref{discrepancy} the \textit{spectral discrepancy} of $\G$ and we refer to the first sum in \eqref{sigma} as the \textit{spectral sign discrepancy} of $\G$ and we denote it by $\sigma(\G)$, that is
\begin{equation} \label{ssd}
\sigma(\G) = \sum_{\lambda \in Sp'(\Gamma), \, |\lambda|\ge 1} sgn(\lambda).
\end{equation}
\end{defi}

We recall from \cite{PV5} that a graph $\G$ is said almost symmetric if $m(\lambda) = m(-\lambda)$ for every $\lambda \ne \lambda_0$ and strongly almost symmetric if in addition $|m(\lambda) - m(-\lambda)|=1$. For instance, it is proved there (see Theorem~3.6) that unitary Cayley sum graphs $G_R^+ = X^+(R,R^*)$, where $R$ is a finite Artinian ring of odd type with $|R|$ odd, are strongly almost symmetric graphs with loops. 

%\goodbreak 

We now consider some special cases of interest from \eqref{sigma}. 

\begin{lem} \label{lemin}
Let $\G$ be a regular graph. 
\begin{enumerate}[$(a)$]
	\item If $\G$ is integral then $\Delta(\G) = m(0) + \sigma(\G)$. \sk 
		
	\item If $\G$ is bipartite (i.e.\@ symmetric) then $\Delta(\G) = m(0) -1 +T+S$. \sk 
	
	\item If $\G$ is integral and bipartite we have $\Delta(\G) = m(0) -1$. \sk 
	
	\item If $\G$ is strongly almost symmetric then $\Delta(\G)= m(0) +T+S$. \sk 
	
	\item If $\G$ is integral and strongly almost symmetric then $\Delta(\G) = m(0)$.
\end{enumerate}
\end{lem}

\begin{proof}
First observe that if $\G$ is integral then $T=S=0$. 
Item $(a)$ follows by \eqref{sigma}, and $(c)$ and $(e)$ are immediate from $(b)$ and $(d)$, respectively.
To prove $(b)$ and $(d)$ recall that $\G$ is bipartite if and only if the spectrum is symmetric. Hence, the first sum in \eqref{sigma} contributes $-1$. In fact, $k$ and $-k$ are eigenvalues with the same multiplicity $m$, and we have removed $\lambda_1=k$ from the summation. In the strongly almost symmetric case, the spectrum is symmetric unless for $\lambda_1=k$. 
Thus, the first sum in \eqref{sigma} equals $0$ in this case.
\end{proof}

From now on, if $\G$ is a graph of $n$ vertices with loops, and $m$ is the maximum number of loops per vertex, the complement $\overline \G$ is taken with respect to the complete graph $K_n$ with $m$-loops added to each vertex, which we denote by $K_n^{*(m)}$.

\begin{prop} \label{lem delta}
Let $\G$ be a $k$-regular graph with $n$ vertices. Let $m\ge 0$ be the maximum number of loops per vertex that $\G$ has. 
Then, $E(\G)=E(\overline \G)$ if and only if
\begin{equation} \label{condition}
n = \begin{cases}
2k+1 - \Delta(\G) & \qquad \text{if $\G$ is loopless $(m=0)$}, \\[1mm]
2k & \qquad \text{if $\G$ has single loops $(m=1)$}, \\[1mm]
2k -(m-1) - \Delta_{m-1}(\G) & \qquad \text{if $\G$ has multiple loops $(m\ge 2)$}.
\end{cases}
\end{equation}
\end{prop}

\begin{proof}
First notice that 
\begin{equation} \label{eq1}
E(\G) = k + \sum_{ \lambda \in Sp'(\G)} |\lambda| \qquad \text{and} \qquad 
E(\overline \G) = \bar k + \sum_{\bar \lambda \in Sp'(\overline \G)} |\bar \lambda|.
\end{equation}

Secondly, if $\G$ has at most $m$ loops per vertex (with $m\ge 0$), the adjacency matrix of $\overline \G$ is exactly 
$J+(m-1)I\!d-A$,
where $I\!d$ is the identity matrix, $A$ is the adjacency matrix of $\G$ and $J$ is the all $1$'s $n\times n$ matrix.
By orthogonality property of eigenvectors, if $k=\lambda_1\ge\lambda_2\ge \cdots\ge \lambda_n$ are the eigenvalues
of $\G$, then the eigenvalues of $\overline \G$ are 
$$\bar k=\bar \lambda_1= n+(m-1)-k \qquad \text{and} \qquad \bar{\lambda}_i=(m-1)-\lambda_i 
\quad \text{for all $i=2,\ldots,n$}.$$

Now, we analyze the different cases. If $m=1$, we have that $\bar \lambda_1= n-k$ and
$|\bar \lambda _i|=|\lambda_i|$ for $i=2,\ldots,n$, thus $E(\G)=E(\overline \G)$ if and only if $n=2k$, by \eqref{eq1}.

On the other hand, if $\G$ has no loops, then $\bar \lambda_1=n-k-1$ and $\bar{\lambda}_i=-1-\lambda_i$ for all $i=2,\ldots,n$, thus we have
$$E(\overline \G) = n-k-1 + 
\sum_{\lambda \in Sp'(\Gamma)} |\lambda +1|.$$
Since by \eqref{delta} we have 
$$|1+\lambda|= |\lambda| + \delta (\lambda),$$
we arrive at 
\begin{equation} \label{eq2}
E(\overline \G) =  n-k-1 + \sum_{\lambda \in Sp'(\G)} |\lambda| + \sum_{\lambda \in Sp'(\G)} \delta(\lambda).
\end{equation}
In this way, by \eqref{eq1} and \eqref{eq2}, we have that $E(\G) = E(\overline{\G})$ if and only if $n=2k+1-\Delta(\G)$, hence \eqref{condition} holds in this case too. 

Finally, if $m\ge 2$, by the second identity in \eqref{eq1} we have 
\begin{eqnarray*} \label{eq3}
E(\overline \G) &=&  n+(m-1)-k + \sum_{\lambda \in Sp'(\G)} |\lambda-(m-1)| \\
 &=& n+(m-1)-k  +  \sum_{\lambda \in Sp'(\G)} |\lambda| + \sum_{\lambda \in Sp'(\G)} \delta_{m-1}(\lambda).
\end{eqnarray*}
In this way, by the first identity in \eqref{eq1}, we have that $E(\G)=E(\overline \G)$ if and only if $n=2k+1-m-\Delta_{m-1}(\G)$ as we wanted to show, and the result follows.
\end{proof}

\begin{rem} \label{no equi}
($i$) In the loopless case, \eqref{condition} holds if and only if $\Delta(\G) = \sigma(\G)+T+m(0)+S$ is an integer, and since $\sigma(\G)$, $T$ and $m(0)$ are integers, \eqref{condition} holds if and only if $S \in \Z$.
Moreover, this in turn happens if and only if 
$\lambda \in \tfrac{1}{2m(\lambda)} \Z$ for every $\lambda \in Sp'(\G) \cap (-1,0)$. 
In particular, if $\G$ has an irrational eigenvalue in $(-1,0)$ then $E(\G) \ne E(\overline \G)$.

\noindent ($ii$) A similar result to Proposition \ref{lem delta} was recently obtained in Corollary 4.3 of \cite{MH}. 
Namely, if $\G$ is a $k$-regular graph of $n$ vertices, $\G$ and $\bar \G$ are equienergetic if and only if 
$$n-k-1 = n^-(\G) - \sum_{\lambda \in Sp'(\G)\cap (-1,0)} (\lambda+1),$$ 
where $n^-(\G)$ denotes the number of negative eigenvalues (we point out here that there is a sign typo in the right hand side in the original expression in \cite{MH}).
\end{rem}

\begin{rem} \label{consequences}
\noindent ($i$) A necessary condition for complementary equienergeticity of $k$-regular graphs without loops can be given. From the inequality in Theorem 3.1 in \cite{MH} in the case of regular graphs we get 
$E(\G)- E(\overline \Gamma) \le 2(k-n+1+n^-(\G))$. Thus, if $\G$ and $\overline \G$ are equienergetic then 
we must have 
\begin{equation} \label{n-}
n^-(\G) \ge n-k-1 = \overline k.
\end{equation}

\noindent ($ii$) For a loopless complementary equienergetic $k$-regular graph $\G$ one can improve upper and lower bounds a little bit. In fact, if we have bounds $f(\G) \le E(\G) \le g(\G)$, for certain functions $f,g$, we have similar bounds for $\overline \G$. Assuming that $E(\G)=E(\overline \G)$ then it holds 
$$\max \{ f(\G), f(\overline \G)\} \le E(\G) \le  \min \{ g(\G), g(\overline \G)\}.$$ 
For instance, we can take the Gutman-Oboudi's lower bound $E(\G) \ge \tfrac{2kn}{k+1}$ for $k$-regular graphs with no eigenvalues in $(-1,1)$ and the Koolen-Moulton's upper bound (see \cite{GuRe} and \cite{KooMo}). In the case of $k$-regular graphs of $n$ vertices we have 
$$\frac{2kn}{k+1} \le E(\G) \le k+\sqrt{(n-1)k(n-k)}.$$
By using the same bounds for $\overline \G$ with regularity degree $n-k-1$, if $\G$ and $\overline \G$ are equienergetic (both with no eigenvalues $|\lambda|<1$), then we have
\begin{equation} \label{new bounds}
\max \big\{ \tfrac{2kn}{k+1}, \tfrac{2(n-k-1)n}{n-k} \big \} \le E(\G) \le \min \big\{ g(k), g(n-k-1) \big\},
\end{equation}
where $g(k)=k+\sqrt{(n-1)k(n-k)}$.
\end{rem}

We now consider some complementary equienergetic graphs with loops. The following is automatic from the previous result.
\begin{coro} \label{comp equi loops}
The set of all complementary equienergetic $k$-regular graphs with single loops are precisely those having $2k$ vertices. 
\end{coro}

Given a graph $\G$ with or without loops, we will denote by $\G^*$ the regular graph obtained by $\G$ by adding all necessary loops (at most one loop per vertex). Hence, $\G$ has only vertices of degree 2 or 3. In this case, if $\G^*$ has $n$ vertices, we take its complement with respect to $K_n^*$.

\begin{exam} \label{loops} 
\noindent ($i$) We consider the following two families of $k$-regular graphs: \textit{cycle graphs with single loops} $\{C_n^*\}_{n\ge 3}$ (3-regular) and \textit{paths with single loops at the ends} $\{P_n^*\}_{n\ge 2}$ \linebreak (2-regular). By \eqref{condition}, the graphs are complementary equienergetic if $n=2k$. Thus, the only graphs equienergetic with their complements in these families are $C_6^*$ and $P_4^*$. However, they are isospectral with their complements. In fact, we have $$Spec(C_6^*)=\{[3]^1,[2]^2, [0]^2, [-1]^1 \} = Spec(\overline{C_6^*}),$$ 
hence $E(C_6^*)=8$. The graphs $C_6^*$ and $\overline{C_6^*}$ are non-isomorphic since $\overline{C_6^*}$ has no loops. 
Moreover, $P_4^*$ is isomorphic to $\overline{P_4^*}$ (so trivially equienergetic and isospectral). 
We have $Spec(P_4^*) = \{[2]^1, [\sqrt 2]^1, [0]^1, [-\sqrt 2]^1\}$ and thus $E(P_4^*)=2(1+\sqrt 2)$.

\noindent ($ii$) We now study regular graphs with single loops of $4$ vertices. By \eqref{condition}, the graphs equienergetic with their complements in this family must be $2$-regular. There are only two possibilities, $\G_1=P_2^* \cup P_2^*$ with complement $C_4$ and $\G_2=P_4^*$ which is self-complementary. In fact, we have 
$$Spec(P_2^* \cup P_2^*)=\{[2]^2, [0]^2\} \qquad \text{ and } \qquad Spec(C_4)=\{[2]^1, [0]^2, [-2]^1\}$$
(hence $P_2^* \cup P_2^*$ and $C_4$ are non-isospectral), and $E(P_2^* \cup P_2^*)=E(C_4)=4$. The graph $P_4^*$ was treated in ($i$).
\hfill $\lozenge$
\end{exam}

\begin{exam}
There are graphs with multiple loops per vertex ($m\ge 2$) which are complementary equienergetic. If $\G$ is a graph we denote by $\G^{**}$ the regular graph obtained from $\G$ by adding one or two loops to $\G$ if it is not regular or the graph with 2 loops added to every vertex if $\G$ is already regular. Denote by $P_0$ the graph with a single vertex. Consider the graph $\G=P_0^{**} \cup P_2^*$. This graph is not connected, $2$-regular with spectrum 
$$Spec(\G)=\{[2]^2, [1]^1, [-1]^1 \}.$$ 
The complement (with respect to $K_4^{**}$) is the graph $\overline{\G}=G^{**}$, where $G$ is the graph with 4 vertices and 4 edges which is not the 4-cycle. The graph $\overline{\G}$ is a connected $3$-regular graph with spectrum 
$$Spec(\overline \G)=\{ [3]^1, [2]^1, [0]^1, [-1]^1 \}.$$ 
Hence $E(\G)=E(\overline \G)=6$.
This is in coincidence with expression \eqref{condition} in the case $m=2$. In fact, for $\G^{**}$ expression \eqref{condition} reads $4 = 3-\Delta_1(\G^{**}) $ and hence 
\begin{equation*}
4 = 3- \sum_{\lambda \ne \lambda_0} \delta_1(\lambda) = 3-\{\delta_1(2)+\delta_1(1)+\delta_1(-1)\}=3-\{(-1)+(-1)+1\}
\end{equation*}
showing that $\G$ is equienergetic with $\overline \G$. \hfill $\lozenge$
\end{exam}
\black

\section{Cubic graphs}
Here, as an application of the results in the previous section, we classify all connected cubic (i.e.\@ $3$-regular or trivalent) graphs which are equienergetic with their complements for two families: ($i$) connected cubic graphs with single loops, ($ii$) integral cubic graphs (loopless). Also, we show that there are no complementary equienergetic cubic loopless graphs at all for two other families: ($iii$) distance-regular cubic graphs and ($iv$) arc-transitive cubic graphs with girth $g<6$ (or exceptional such graphs with $g=6$).

We will use the list of connected non-isomorphic graphs of 6 vertices given in \cite{CP}. There are 112 such graphs. We will denote by $\G_n$ the graph numbered $n$ in Table 1 in \cite{CP} and by $\G_n^*$ the regular graph with loops obtained from $\G_n$.
 
\begin{prop} \label{cubic with loops}
Up to isomorphisms, there are only 5 pairs of connected cubic graphs with single loops equienergetic (non-isomorphic) with their complements. The complementary equienergetic pairs are 
$$\{\G_{51}, \G_{106}^*\}, \{\G_{69}^*, \G_{93}^*\}, \{\G_{70}^*, \G_{92}^*\}, \{\G_{72}^*, \G_{89}^*\}, \{\G_{74}^*, \G_{84}^*\}.$$ The graphs in each pair are non-isospectral. The graphs in the first and third pairs are equienergetic.  
\end{prop}

\begin{proof}
 By \eqref{condition} in Proposition \ref{lem delta}, the complementary equienergetic $3$-regular graphs with loops must have 6 vertices. By using the list of non-isomorphic graphs of 6 vertices in Table 1 in \cite{CP}, and considering only those graphs $\G$ having only degree 2 or 3, by taking $\G^*$ we get all $3$-regular graphs with loops. 
 
 There are 11 graphs in this list having degrees 2 or 3 only, namely $\G_{51}=C_3 \otimes C_3$, $\G_{52}=K_{3,3}$ (which are $3$-regular) and $\G_{69}$, $\G_{70}$, $\G_{72}$, $\G_{74}$, $\G_{84}$, $\G_{89}$, $\G_{92}$, $\G_{93}$ and $\G_{106}$. So, we consider the graphs $\G_{51}$, $\G_{52}$, $\G_{69}^*$, $\G_{70}^*$, $\G_{72}^*$, $\G_{74}^*$, $\G_{84}^*$, $\G_{89}^*$, $\G_{92}^*$, $\G_{93}^*$ and $\G_{106}^*$. 
 
  One can check that the complement of $\G_{52}$ (with respect to $K_6^*$) is $C_3^* \cup C_3^*$, which is not connected, and hence we discard this graphs. Also, analyzing the complements of the remaining ten graphs we see that we have the following 5 pairs of graphs and their complements: $\{\G_{51}, \G_{106}^*\}$, $\{\G_{69}^*, \G_{93}^*\}$, $\{\G_{70}^*, \G_{92}^*\}$, $\{\G_{72}^*, \G_{89}^*\}$ and $\{\G_{74}^*, \G_{84}^*\}$.
 
In Table 1 we give the graphs, their spectra and the energies.   
\begin{table}[H]
\caption{Complementary equienergetic connected cubic graphs with loops}
	\begin{tabular}{|c|c|c|}
		\hline
		graph & spectrum & energy \\ \hline
		$\G_{51}$ & $\{ [3]^1, [1]^1, [0]^2, [-2]^2 \}$ & $8$ \\
		$\G_{106}^*$ & $\{ [3]^1, [2]^2, [0]^2, [-1]^2 \}$ & 8 \\ 
		\hline
		$\G_{69}$ 	& $\{ [3]^1, [\sqrt 3]^1, [1]^1, [-1]^2, [-\sqrt 3]^1 \}$ & $6+2\sqrt 3 \simeq 9,4641$ \\
		$\G_{93}^*$ & $\{ [3]^1, [\sqrt 3]^1, [1]^2, [-1]^1, [-\sqrt 3]^1 \}$ & $6+2\sqrt 3 \simeq 9,4641$\\ 
		\hline
		$\G_{70}$ & $\{ [3]^1, [2]^1, [0]^2, [-1]^1, [-2]^1 \}$ & 8 \\
		$\G_{92}^*$ & $\{ [3]^1, [2]^1, [1]^1, [0]^2, [-2]^1\}$ & 8 \\ 
		\hline
		$\G_{72}$ & $\{ [3]^1, [\sqrt 2]^1, [1]^1, [0]^1, [-\sqrt 2]^1, [-2]^1 \}$ & $6+2\sqrt 2 \simeq 8,8242$ \\
		$\G_{89}^*$ & $\{ [3]^1, [2]^1, [\sqrt 2]^1, [0]^1, [-1]^1, [-\sqrt 2]^1 \}$ & $6+2\sqrt 2 \simeq 8,8242$ \\ 
		\hline
		$\G_{74}$ & $\{ [3]^1, [\frac{\sqrt{17} -1}2]^1, [0]^3, [\frac{-1-\sqrt{17}}2]^1 \}$ & $3+\sqrt{17} \simeq 7,1231$ \\
		$\G_{84}^*$ & $\{ [3]^1, [\frac{\sqrt{17} +1}2]^1, [0]^3, [-\frac{1-\sqrt{17}}2]^1 \}$ & $3+\sqrt{17} \simeq 7,1231$ \\ 
		\hline
	\end{tabular}
\end{table}
The remaining assertions are clear from the previous table. 
\end{proof}

From now on, we will only consider loopless graphs in the paper with the exception of the last result in Proposition \ref{GR+}.

We recall that the Cartesian product of two graphs $\G_1=(V_1,E_1)$ and $\G_2=(V_2,E_2)$, denoted by $\G_1 \square \G_2$, is the graph with vertex set $V=V_1 \times V_2$ and $(v_1,v_2) \sim (w_1,w_2)$ if and only if $v_1=w_1$ and $v_2 \sim w_2$ or $v_1 \sim w_1$ and $v_2=w_2$.

\begin{prop} \label{integral cubic}
Up to isomorphisms, the only connected integral complementary equienergetic cubic graphs are the 3-prism $K_3 \square K_2$ (non-bipartite) with $E(K_3\square K_2)=8$ and the cube graph $Q_3$ (bipartite) with $E(Q_3)=12$.
Moreover, the graphs in the pairs $\{K_3\square K_2, \overline{K_3\square K_2}\}$ and $\{ Q_3, \overline Q_3 \}$ are mutually non-isospectral. 
\end{prop}

\begin{proof}
The classification of connected integral cubic graphs were done independently by Bussemaker and Cvetkovic \cite{BC} and Schwenk \cite{Sch} in 1976. There are only 13 such graphs, 8 of them are bipartite. % (listed first in the table below). 
The graphs and their spectra are given in Table 2 (bipartite graphs are listed first),
taken from the Brouwer-Haemers' book \cite{BH}. 
\begin{table}[h]
\caption{All connected integral cubic graphs}
$$\begin{tabular}{|c|c|c|c|c|}
\hline
\# & graph & name & $n$ & spectrum  \\ \hline
1 & $K_{3,3}$ 		&		& 6 & $\pm 3, 0^4$  \\ 
2 & $Q_3=2^3$ 		&	cube & 8 & $\pm 3, (\pm 1)^3$  \\ 
3 & $K_{2,3}^* \otimes K_2$ && 10& $\pm 3, \pm 2, (\pm 1)^2, 0^2$  \\ 
4 & $C_6 \square K_2$ 		&& 12& $\pm 3, \pm 2, \pm 1, 0^4$  \\ 
5 & $\Pi \otimes K_2$ 		&Desargues & 20& $\pm 3, (\pm 2)^4, (\pm 1)^5$  \\
6 & $T^* \otimes K_2$ 		&& 20& $\pm 3, (\pm 2)^4, (\pm 1)^5$  \\ 
7 & $\Sigma \otimes K_2$ 	&& 24& $\pm 3, (\pm 2)^6, (\pm 1)^3, 0^4$  \\ 
8 & $GQ(2,2)$ 				& Tutte-Coxeter & 30& $\pm 3, (\pm 2)^9, 0^{10}$  \\ 
\hline
9  & $K_4$ 					&& 4& $3, (-1)^3$  \\ 
10 & $K_3 \square K_2$ 		&& 6& $3, 1, 0^2, (-2)^2$  \\ 
11 & $\Pi$ 					& Petersen & 10& $3, 1^5, (-2)^4$  \\ 
12 & $(\Pi \otimes K_2) / \sigma$ && 10& $3, 2, 1^3, (-1)^2, (-2)^3$  \\ 
13 & $\Sigma$ 				&& 12& $3, 2^3, 0^2, (-1)^3, (-2)^3$  \\ 
\hline
\end{tabular}
$$	
(The Tutte-Coxeter graph is also known as the Levi graph or Tutte's 8-cage.)
\end{table}
Here, if $G$ is a graph with degrees 2 or 3, $G^*$ denotes the 3-regular graph obtained from $G$ adding loops to the vertices of degree 2 (see the rest of notations in \cite{BH}).

Let us see that no graph $\G$ from the table, except for the graphs $Q_3=2^3$ and $K_3\square K_2$, is equienergetic with its complement. Note that although $G^*$ has loops, $G^* \otimes K_2$ is loopless by definition of the Kronecker product.
Thus, since $\G$ is loopless and 3-regular, the condition $E(\G)=E(\overline \G)$ is, by Proposition \ref{lem delta}, equivalent to
\begin{equation} \label{cubics}
\Delta(\G) = 7-n.
\end{equation}
Now, if $\G$ is bipartite (\#1--\#8), by ($c$) in Lemma \ref{lemin} we have that $\G$ and $\overline \G$ are equienergetic if and only if 
$$m(0) = 8-n.$$ 
Since $m(0)\ge 0$ and $8-n <0$ for $n\ge 10$ we only have to check this condition for the graphs \#1 and \#2. 
For $K_{3,3}$ we have $4\ne 2$ while for $Q_3$ we have $0=0$. Hence $Q_3$ is equienergetic with its complement. 
In fact, $\overline Q_3$ has spectrum $\{ [4], [2], [0]^3, [-2]^2\}$ and $E(Q_3)=E(\overline Q_3)=12$. Furthermore, $Q_3$ and $\overline Q_3$ are non-isospectral.

Now, if $\G$ is non-bipartite (\#9--\#13), condition \eqref{cubics} for equienergy now reads
$$ 7-n = \sigma(\G) + m(0)$$
by ($a$) in Lemma~\ref{lemin}. 
One can check that this equation is only satisfied by graph \#10, i.e.\@ $K_3 \square K_2$. Also, we have
$Spec(\overline{K_3 \square K_2}) = \{ [2], [1]^2, [-1]^2, [-2] \}$. Therefore, $E(K_3 \square K_2) = 
E(\overline{K_3 \square K_2}) = 8$ and $K_3 \square K_2$ and its complement are non-isospectral, thus concluding the proof.
\end{proof}

\begin{coro} \label{DRG cubic}
There are no distance-regular nor distance-transitive cubic graphs equienergetic with their own complement.
\end{coro}

\begin{proof}
There are 13 finite distance-regular cubic graphs. In 1971, Biggs and Smith (\cite{BS}) classified the finite distance-transitive cubic graphs (hence also distance-regular), and found 12. The remaining one, the Tutte's 12-cage or Benson graph (which is not distance-transitive), was found 15 years later by Biggs et al (\cite{BBS}). The graphs and their spectra are given in Table 3. 
\begin{table}[h!]
\caption{All finite distance-regular cubic graphs}	
$$\begin{tabular}{|c|c|c|c|c|}
	\hline
	\# & graph & $n$ & spectrum & dist.-trans. \\ \hline
	1 & $K_4$ (tetrahedron)		& 4 & $3, -1^3$ & yes \\ 
	2 & $K_{3,3}$ (utility)  	& 6 & $\pm 3, 0^4$  & yes \\ 
	3 & $Q_3$ (cube) 			& 8 & $\pm 3, (\pm 1)^3$  & yes\\ 
	4 & Petersen 			& 10& $3, 1^5, (-2)^4$ & yes \\  
	5 & Heawood 	& 14 & $\pm 3, (\pm \sqrt 6)^6$& yes \\ 
	6 & Pappus 		& 18 & $\pm 3, (\pm \sqrt 3)^6, 0^4$ & yes \\
	7 & dodecahedron & 20 & $3, (\pm \sqrt 5)^3, 1^5, 0^4, (-2)^4$ & yes \\ 
	8 & Desargues & 20& $\pm 3, (\pm 2)^4, (\pm 1)^5$ & yes \\ 
	9  & Coxeter & 28 & $3, (1+\sqrt 2)^6, 2^8, (1-\sqrt 2)^6,  (-1)^7$ & yes \\ 
	10 & Tutte-Coxeter %or Levi 
	& 30 & $\pm 3, (\pm 2)^9, 0^{10}$ & yes \\ 
	11 & Foster & 90 & $\pm 3, (\pm 2)^9, (-1)^{18}, 0^{10}$ & yes \\ 
	12 & Biggs-Smith &102& $3, \lambda_1^9, 2^{18}, \lambda_2^{16}, 0^{17}, \lambda_3^{16}, \lambda_4^{9}, \lambda_5^{16}$ & yes \\ 
	13 & Tutte's 12-cage & 126 & $\pm 3, (\pm \sqrt 6)^{21}, (\pm \sqrt 2)^{27}, 0^{28}$ & no \\ 
	\hline
	\end{tabular}
	$$
Here, $x^2-x-4 = (x-\lambda_1)(x-\lambda_4)$ and $x^3+3x^2-3=(x-\lambda_2)(x-\lambda_3)(x-\lambda_5)$. 
So $\lambda_1=2.562..$, $\lambda_2=0.879...$, $\lambda_3=-1.347...$, $\lambda_4=-1.562...$ and $\lambda_5=-2.532...$
\end{table}

As in the proof of the previous corollary, we have to check if \eqref{cubics} is satisfied. The graphs $K_{3,3}$, $Q_3$, Petersen, Desargues and Tutte-Coxeter were checked in the previous proof.  
The remaining cases can be checked analogously and we omit the details.
Note that both the Coxeter graph and the Biggs-Smith graph are not equienergetic with their complements by Remark \ref{no equi}, since they have an irrational eigenvalue inside $(-1,0)$.
\end{proof}

\begin{coro} \label{arc trans g6}
There are no arc-transitive cubic graphs of girth $g<6$ or exceptional arc-transitive cubic graphs with $g=6$ equienergetic with their complements. 
\end{coro}

\begin{proof}
It is well-known that there only 5 arc-transitive cubic graphs with $g<6$, namely $K_4$, $K_{3,3}$, the cube $Q_3$, the Petersen graph $\Pi$ and the dodecahedron graph. They were all discarded in Proposition \ref{integral cubic} and Corollary \ref{DRG cubic}. 
The exceptional arc-transitive cubic graphs with $g=6$ are (\cite{KM}) Heawood graphs, the Pappus graph, the Desargues graph and the Möbius-Kantor graph. The first three graphs were discarded in Corollary \ref{DRG cubic}. The Möbius-Kantor graph $\Gamma_{MK}$ has 16 vertices with spectrum 
$$Spec(\G_{MK}) = \{[3], [\sqrt 3]^4, [1]^3, [-1]^3, [-\sqrt 3]^4, [-3]\}.$$ 
One can check that \eqref{condition} does not hold and hence by Proposition \ref{lem delta} the graph $\G_{MK}$ is not equienergetic with its complement. 
\end{proof}

Other families of cubic graphs can be considered, such as vertex-transitive cubic graphs, edge-transitive cubic graphs and symmetric (vertex and edge transitive) cubic graphs. However, this families are only classified up to some number of vertices $N$ (different in each case). 
Nevertheless, one could look for examples of complementary equienergetic (non-isospectral) graphs in these families or classify all such graphs up to $N$.

\section{Bipartite regular graphs}
Here we characterize all bipartite graphs which are equienergetic with their own complements. 
We recall that a graph is said bipartite if it has a bipartition of its vertices such that vertices in the same partition are not neighbors. This automatically implies that bipartite graphs are loopless.
There are some equivalences such as: $\G$ is a bipartite graph if and only if
$\G$ has chromatic number $\chi(\G)=2$ or if and only if $\G$ has only even length cycles.
In the case of regular bipartite graphs there are easy spectral conditions: if $\G$ is $k$-regular, then $\G$ is bipartite if and only if  $\mathrm{Spec} (\G)$ is symmetric, or 
more generally, if and only if $-k\in \mathrm{Spec} (\G)$.

The most common bipartite regular graph is the complete bipartite graph $K_{t,t}$ which is a graph with $2t$ vertices and $(t,t)$-bipartition such that any vertex of one part is connected with all of the vertices of the other part, hence $t$-regular.
Another important example of bipartite regular graph is the \textit{crown graph} $Cr(t)$, obtained from $K_{t,t}$ by deleting any perfect matching (hence $(t-1)$-regular). The crown graph can also be obtained in many other ways, for instance as the Kronecker product of the complete graphs $K_2 \otimes K_t$.

In the following result we show that the only bipartite regular graphs which are equienergetic and non-isospectral with their complements are essentially the crown graphs.

\begin{thm} \label{Teo Bip equien}
Let $\G$ be a regular bipartite graph. 
Then $\G$ and $\overline \G$ are equienergetic if and only if $\G=C_4$ or 
$\G=Cr(t)$ with $t\ge 2$. In this case $E(C_4)=4$ and $E(Cr(t))=4(t-1)$.
Moreover, $\G$ and $\overline \G$ are both integral and non-isospectral to each other. 
\end{thm}

\begin{proof}
We first show that crown graphs and their complements are equienergetic non-isospectral graphs.	
Let $\G=Cr(t)$ with $t\ge 2$. For $t=2$,  $Cr(2)$ is the disjoint union of two copies of $K_2$, 
	and so we have that
	\begin{equation} \label{spec2K2}
	Spec(Cr(2)) = Spec (2K_2) = \{ [1]^2, [-1]^2 \}.
	\end{equation}
	The complement of $Cr(2)$ is $C_4$, the cycle of length $4$, then
	$$ Spec(\overline{Cr(2)}) = Spec (C_4) = \{ [2]^1,[0]^2, [-2]^1 \}.$$
	Clearly $Cr(2)$ and $C_4$ are equienergetic with $E(Cr(2))=E(C_4)=4$ and non-isospectral. 

In general, for any $t\ge 2$, since $Cr(t)=K_{2}\otimes K_t$, the spectrum of this graph is 
$$Spec (Cr(t)) = \{ [t-1]^{1}, [1]^{t-1}, [-1]^{t-1}, [-(t-1)]^{1} \}$$
(note that for $t=2$ we have \eqref{spec2K2}).
Since $Cr(t)$ is $k$-regular with $k=t-1$ and has $n=2t$ vertices, by ($c$) in Lemma \ref{lemin} we have
	$$-1 = m(0)-1 = \Delta(\G) = 2k+1-n = -1$$
and thus, by Proposition \ref{lem delta}, the graphs $\G$ and $\bar \G$ are equienergetic, with energy 
$E(\G) = E(\overline \G) = 4(t-1)$.

Notice that $\overline \G$ is non-bipartite for $t\ge 3$, since it has a complete graph of $t\ge3$ vertices as induced subgraph, so $\overline \G$ has at least one odd length cycle. This implies that $\overline \G$ has no symmetric spectra. 
Therefore, the graphs $\G$ and $\overline \G$ are non-isospectral.
	
Now, we show that if $\G$ is a bipartite regular graph, equienergetic with its complement, then it must be a crown graph or $C_4$. 
Assume that $\G$ is a bipartite $k$-regular graph (hence loopless) with $n=2t$ vertices, equienergetic with $\overline \G$. Then, by item $(b)$ of Lemma~\ref{lemin}, we have that 
	\begin{equation}\label{delta T}
	\Delta(\G) = m(0)-1+T+S,
	\end{equation}
Since $\G$ has symmetric spectrum, we have that 
	$$T=\# (Sp'(\G) \cap (0,1)) = \# (Sp'(\G) \cap (-1,0))$$ 
and thus
	$$T+S = \sum_{\lambda\in Sp'(\G) \cap (-1,0)} 1+\sum_{\lambda\in Sp'(\G) \cap (-1,0)} (2\lambda+1) =
2 \sum_{\lambda\in Sp'(\G) \cap (-1,0)} (\lambda+1) \ge 0.$$
Using that $\G$ and $\overline \G$ are equienergetic, by Proposition \ref{lem delta} and \eqref{delta T}, we have that 
	$$k= \tfrac 12 (n-2+m(0)+T+S) = t-1 + \tfrac 12 (m(0)+T+S).$$
Now, $m(0)+T+S$ must be a non-negative even integer and $k \le t$ leave us with only two possibilities; either 
$m(0)+T+S=0$, and hence $k=t-1$, or else $m(0)+T+S=2$, and thus $k=t$. 
	
Assume first that $k=t$. In this case $\G$ is the complete bipartite graph $K_{t,t}$ with complement  
the disjoint union of two complete graphs of $t$ vertices $\overline K_{t,t} = 2K_t$.
Since 
	$$Spec(K_{t,t}) = \{ [t]^1, [0]^{2t-2}, [-t]^1 \} \quad \text{and} \quad Spec(2K_{t}) = \{[t-1]^2,[-1]^{2t-2}\}, $$
we have that $E(K_{t,t})=2t$ and $E(2K_t)=4t-4$. 
Clearly, $E(K_{t,t})=E(2K_{t})$ if and only if $t=2$, i.e.\@ when $\G$ is the cycle of length $4$. 
	
Finally, if $k=t-1$, there is only one bipartite $t-1$ regular graph with $n=2t$ vertices, which is obtained from the  complete bipartite graph $K_{t,t}$ by removing a perfect matching. This complete the proof. 
\end{proof}

\begin{rem}
($i$) The equienergetic pair of non-isospectral graphs $\{ C_4, K_2 \otimes K_2\}$ is well-known. As we showed in the theorem and its proof, they are  the first pair of an infinite family of pairs of equienergetic non-isospectral graphs, the only possible one for bipartite regular graphs. It is also known that $C_6$ and $\overline C_6$ are equienergetic. Note that $C_6=Cr(3)$.
The cube graph $Q_3$ from Proposition \ref{integral cubic}, which is bipartite, is isomorphic to $Cr(4)$.

\noindent ($ii$) Note that the only connected regular bipartite graph with $n\ge 4$ whose complement is also bipartite is $C_4$.
\end{rem}

As a direct consequence, we obtain the following result.
\begin{coro}
For any $t\ge 2$, there exist regular graphs of $2t$ vertices equienergetic and non-isospectral with their own complements.
\end{coro}

\section{General strongly regular graphs}
Here we will characterize the parameters of strongly regular graphs which are complementary equienergetic. 
For a primitive strongly regular graph we will give a simple criterion to decide whether it is complementary equienergetic 
or not.

Let $\G$ be a regular graph that is neither complete nor empty. Then $\G$ is
said to be \textit{strongly regular} with parameters $(n, k, e, d)$
if it is $k$-regular with $n$-vertices, every pair of adjacent vertices has $e$ common neighbors,
and every pair of distinct nonadjacent vertices has $d$ common neighbors.
If $\G$ is strongly regular with parameters
$(n, k, e, d)$, then its complement $\bar \G$ is also strongly regular with parameters
$(n, \bar k, \bar e, \bar d)$, where $\bar k = n-k-1$, 
$$\bar e= n-2-2k+d \qquad \text{and} \qquad \bar d= n-2k+e.$$

A strongly regular graph $\G$ is called \textit{primitive} if both $\G$ and $\overline \G$ are connected, otherwise is called \textit{imprimitive}. It is known that the only connected strongly regular graph which is 
imprimitive is the complete multipartite graph $K_{a \times m}$ of $a$ parts of size $m$ with parameters 
$$srg(am, (a-1)m, (a-2)m,(a-1)m),$$ 
with complement the disjoint union of copies of the complete graphs $\overline{K_{m\times m}} = aK_m$ with parameters 
$$srg(am, m-1,m-2, 0).$$

The following proposition shows that the only connected imprimitive strongly regular graphs which are equienergetic with their complements are the complete multipartite graphs with the same number of parts than the size of the parts.

\begin{prop} \label{prop imprimitive}
	Let $\G$ be a connected imprimitive strongly regular graph.
	Then, $\G$ and $\overline{\G}$ are equienergetic if and only if $\G=K_{m\times m}$ for some $m>1$. 
	In this case $\G$ and $\bar \G$ are not isospectral. Moreover, we have $E(K_{m\times m})=2(m-1)m$ and hence $4\mid E(K_{m \times m})$.
\end{prop} 

\begin{proof}
Let $\G$ be an imprimitive connected strongly regular graph. Then $\G=K_{a\times m}$ for some integers $a,m\ge 2$. 
The spectrum of $\G$ is
\begin{equation} \label{spec Kmm}
\mathrm{Spec} (\G) = \{ [(a-1)m]^{1}, [0]^{a(m-1)}, [-m]^{a-1} \}.
\end{equation}
Thus, by ($a$) in Lemma \ref{lemin}, $\G$ satisfies  
	$$ \Delta (\G) = m(0)+ \sigma(\G) = a(m-1)-(a-1) = am - 2a + 1.$$
Since $\G$ has $n=am$ vertices and regularity degree $k=(a-1)m$, Proposition \ref{lem delta} implies that
	$\G$ and $\bar{\G}$ are equienergetic if and only if
	$$am-2a+1 = \Delta(\G) = 2k+1-n = 2(a-1)m+1-am = am-2m+1,$$
	which clearly holds if and only if $a=m$. In this case, $\G$ and $\bar \G$ are not isospectral since they have different degree of regularity. The remaining assertions are straightforward.
\end{proof}

\begin{rem}
Recently, in \cite{GuRe}, the authors give the following lower bound for the energy of a $t$-regular graph $\G$ of $n$ vertices without eigenvalues in the interval $(-1,1)$,
\begin{equation} \label{bound}
E(\G) \ge \tfrac{2tn}{t+1},
\end{equation}
with equality if and only if $\G$ is the complete graph $K_{t+1}$ or the crown graph $Cr(t+1)$.
This is in coincidence with Theorem \ref{Teo Bip equien}. Note that the complete bipartite graph $K_{m \times m}$ in the previous proposition also satisfies the equality in \eqref{bound}, but has $0$ as one of its eigenvalues (see \eqref{spec Kmm}).
However, 
inequality \eqref{bound} cannot be extended to all regular graphs without the restriction of having no eigenvalues in $(-1,1)$, since the graph $K_3 \square K_2$ given in Table 1 (see the proof of Proposition \ref{integral cubic}) has eigenvalues in $(-1,1)$ and the inequality does not hold.  
Also, notice that \eqref{bound} cannot be extended for regular graphs with loops without the restriction of no eigenvalues in $(-1,1)$.  In fact, $C_6^*$ is $3$-regular complementary equienergetic with $6$ vertices and energy $8$ having $0$ as an eigenvalue (see Example \ref{loops}), hence \eqref{bound} would read $8\ge 9$, which is absurd. 
\end{rem}

We now give a simple necessary and sufficient condition, in terms of the parameters, for a primitive strongly regular graph $srg(n,k,e,d)$ to be equienergetic with its complement. We will need the following notation
\begin{equation} \label{Delta}
\alpha = (e-d)^2+4(k-d).
\end{equation}

\begin{prop} \label{srg equien}
	Let $\G$ be a primitive strongly regular graph with parameters $srg(n,k,e,d)$. 
	Then, $\G$ and $\overline{\G}$ are equienergetic if and only if
	\begin{equation} \label{cond Equien srg}
	 n = \frac{2k(\sqrt{\alpha}+1)}{\sqrt{\alpha}-(e-d)}+1.	
	\end{equation}
\end{prop}

\begin{proof}
It is well-known that any connected strongly regular graph have exactly three eigenvalues $k, r, s$, 
%{\blue with $k>r>0>s$,} 
where the non trivial eigenvalues are given by 
	\begin{equation} \label{eig srg}
	r = \tfrac 12 (e-d + \sqrt{\alpha})  \qquad \text{and} \qquad s = \tfrac 12 (e-d - \sqrt{\alpha})
	\end{equation}
with multiplicities
	\begin{equation} \label{mult srg}
	m_r = \tfrac{n-1}2 - \tfrac{2k+(n-1)(e-d)}{2\sqrt{\alpha}}  \qquad \text{and} \qquad 
	m_s = \tfrac{n-1}2 + \tfrac{2k+(n-1)(e-d)}{2\sqrt{\alpha}} .
	\end{equation}
	Since $\G$ is primitive we have that 
	$d<k$, and thus $\sqrt{\alpha}>|e-d|$. 
	On the other hand,  $\sqrt{\alpha} > e-d+2$, since this inequality is equivalent to $k-e>1$ which is trivially true.
	These observations imply that
    $k> r > 0$ and $s < -1$. By \eqref{TS}, we have $T=S=m(0)=0$.
	By Proposition \ref{lem delta}, $E(\G)=E(\overline \G)$ if and only if $\Delta(\G)=2k+1-n$. 
	By \eqref{sigma}, we have 
	\begin{equation} \label{delta srg}
		\Delta(\G) = \sum_{\lambda \in Sp'(\G)} \delta(\lambda) = m_r - m_s = -\tfrac{2k+(n-1)(e-d)}{\sqrt{\alpha}}.	
	\end{equation} 
After routine computations we obtain that $\Delta(\G) = 2k+1-n$ if and only if \eqref{cond Equien srg} holds. 
\end{proof}

\noindent
\textit{Note:} This proposition appeared recently in \cite{RPPAG} with a similar proof, using the expressions for the energies. 
Our proof is slightly simpler, and obtained as a direct consequence of a the more general fact given in Proposition \ref{lem delta}. 

Using the previous proposition we now characterize the parameters of all primitive strongly regular graphs which are equienergetic with their complements. We recall that a \textit{conference graph} is any strongly regular graph with parameters 	
	\begin{equation} \label{conf}
		srg(4t+1,2t,t-1,t).
	\end{equation}
It is well-known that they exist if and only if $n=4d+1$ is the sum of two perfect squares, and that if $\G$ is a strongly regular graph which is not a conference graph, then it is integral.

\begin{thm} \label{Teo equien srg e le d}
Let $\G=srg(n,k,e,d)$ be a primitive strongly regular graph.
	Then, $\G$ and $\overline \G$ are equienergetic if and only if one of the following 3 cases occur (provided they exist): 
	\begin{enumerate}[$(a)$]
		\item If $e-d=-1$, then $\G= \G_d = srg(4d+1,2d,d-1,d) \simeq \overline \G$ is a conference graph with $d\ge 1$. In this case, 
		$E(\G)= 2d (1+ \sqrt{4d+1})$. 
		\msk 

		\item If $e-d=2h$ with $h \in \Z$, then $\G = \G_{h,\ell}=srg(4\ell^2, k,d+2h,d)$ with 
		$$k=(\ell-h)(2\ell-1) \qquad \text{and} \qquad d=(\ell-h)(\ell-h-1),$$
where $\sqrt{\alpha}=2\ell+1$ for some $\ell\in \mathbb{N}$ such that $\ell \notin \{ \pm h, \pm(h+1) \}$. 
In this case, $E(\G)=2(\ell-h)(2\ell-1)(\ell+h+1)$. \msk 

		\item  If $e-d=2h-1$ for $h \in \Z \smallsetminus \{0\}$, then $\G=\G_{h,\ell}'=srg((2\ell+1)^2, k,d+2h-1,d)$ with 
		$$k=2\ell(\ell-h+1) \qquad \text{and} \qquad d=(\ell-h)(\ell-h+1),$$
where $\sqrt{\alpha}=2\ell$ for some $\ell\in \mathbb{N}$ such that $\ell \notin \{ \pm h, -(h+1), h-1 \}$. In this case $E(\G)=4\ell(\ell-h+1)(\ell+h+1)$.
	\end{enumerate}
Moreover, if $\G$ is not a conference graph, $E(\G)$ is divisible by 4.
\end{thm}

\begin{proof}
($a$) 
Assume first that $e-d=-1$. By Proposition \ref{srg equien}, $\G$ and $\overline \G$ are equienergetic if and only if 
	$$n = \frac{2k(\sqrt{\alpha}+1)}{\sqrt{\alpha}-(e-d)}+1= 2k+1.$$
	Now, recall that the parameters of any $srg(n,k,e,d)$ satisfy the relation 
\begin{equation} \label{constrains}
k(k-e-1)=d(n-k-1).
\end{equation}
Since $k\ne 0$, we have that $k-1-e=d$ and thus $k=2d$. Therefore, $\G$ is strongly regular with parameters
	$(4d+1,2d,d-1,d)$, i.e.\@ $\G$ is a conference graph. It is known that conference graphs 
	are self-complementary, so that $\G$ and $\overline \G$ are trivially equienergetic. The energy is given by
	$$E(\G) = k + m_r r + m_s |s| = 2d+ 2d(r+|s|) = 2d(1+\sqrt{4d+1}),$$ 
	where we have used \eqref{eig srg} and \eqref{mult srg}.

\sk 
($b$) Suppose that $e-d=2h$ for some integer $h$. 
	Since $\G$ is not a conference graph, the spectrum of $\G$ is integral and hence, from \eqref{eig srg}, $\alpha$ is a perfect square, say $\alpha = a^2$ with $a \in \N$. Thus, by \eqref{Delta} we have 
	$$k = \tfrac 14 (a^2-4h^2) + d.$$ 
	Since $k$ is an integer, then $a=2\ell$ for some $\ell \in \N$ and thus
	$$k=(\ell-h)(\ell+h)+d.$$
	Notice that if $\ell=\pm h$, then $k=d$ which cannot occur since $\G$ is primitive by hypothesis.
	So $\sqrt{\alpha}-(e-d) = 2(\ell-h) \ne 0$.
	Now, if $\G$ and $\overline \G$ are equienergetic, by Proposition \ref{srg equien} we have that
\begin{equation} \label{nk}
n = \tfrac{2k(\sqrt{\alpha}+1)}{\sqrt{\alpha}-(e-d)}+1 = \tfrac{k(2\ell+1)}{\ell-h} + 1.
\end{equation}	

Now, on the one hand we have 
\begin{eqnarray*}
k-e-1 &=& (\ell^2-h^2+d)-(d+2h)-1 \\ 
      &=&  \ell^2-(h+1)^2=(\ell-h-1)(\ell+h+1), 
\end{eqnarray*}	
and on the other hand, by \eqref{nk}, we have 
	$$n-k-1 = \tfrac{k(2\ell+1)}{\ell-h}+1-k-1 = \tfrac{k(\ell+h+1)}{\ell-h}.$$
Putting together these last two expressions in \eqref{constrains}, and canceling $k\ne 0$ on both sides, we get 
	$$(\ell-h)(\ell-h-1)(\ell+h+1)=d(\ell+h+1).$$
If $\ell=-(h+1)$, then $\G$ has $e+2$ vertices and regularity degree equal to $e+1$. But the only regular graph with these parameters is the complete graph, which is not primitive. 
Since $\G$ is primitive, we must have $\ell\ne-(h+1)$ and thus 
	$$d=(\ell-h)(\ell-h-1).$$
Notice that in this case if $\ell=h+1$, then $d=0$ which cannot occur by primitivity of $\G$.
Therefore we get $k=(\ell-h)(2\ell-1)$ and $n=(2\ell-1)(2\ell+1)+1=4\ell^2$ with $\ell \not\in\{\pm h,\pm (h+1)\}$ as asserted.

Finally, by taking into account that $0<r=\frac{1}{2}(e-d+\sqrt{\alpha})=\ell+h$ with multiplicity $m_r=(2\ell-1)(\ell-h) $ and $-1\ge s=\frac{1}{2}(e-d-\sqrt{\alpha})=h-\ell\le -1$ with $m_s=(2\ell-1)(\ell+h+1)$,
the energy of $\G$ is given by
$$E(\G) = k + r m_r - s m_s = 2(\ell-h)(2\ell-1)(\ell+h+1).$$
	
\sk 
($c$) Now, suppose that $e-d=2h-1$ for some $0\ne h\in\Z$. 
	Hence, $\G$ is not a conference graph and thus $\alpha=a^2$ with $a \in \N$, so we obtain that
	$k=\tfrac 14(a^2-(2h+1)^2)+d \in \Z$. 
Then $a$ is odd, say $a=2\ell+1$ for some $\ell \in \N$, and thus
	$$k = (\ell^2+\ell)(h^2-h)+d = (\ell-h+1)(\ell+h)+d.$$
We have that $\ell \ne h-1$ and $\ell \ne -h$ (otherwise $k=d$, which cannot occur since $\G$ is primitive by hypothesis) and thus  $\sqrt{\alpha}-(e-d) = 2(\ell-h+1) \ne 0$.
If $\G$ and $\bar \G$ are equienergetic, by Proposition \ref{srg equien} we have that
	$$n = \tfrac{2k(\sqrt{\alpha}+1)}{\sqrt{\alpha}-(e-d)}+1 = \tfrac{2k(\ell+1)}{\ell-h+1}+1.$$ 
As in ($b$), we have 
\begin{eqnarray*}
k-e-1 &=& (\ell-h+1)(\ell+h)+d-(d+2h-1)-1 \\ 
&=& \ell^2-h^2+\ell-h = (\ell-h)(\ell+h-1).
\end{eqnarray*}	
and also
$$ n-k-1 = \tfrac{2k(\ell+1)}{\ell-h+1}-k = \tfrac{k(\ell+h+1)}{\ell-h+1}.$$
Using the previous equations in \eqref{constrains} 
and canceling $k\ne 0$ we get  
$$(\ell-h+1)(\ell-h)(\ell+h+1)=d(\ell+h+1).$$
As in part ($b$), 	
since $\G$ is primitive we obtain that $\ell \ne -h-1$ and thus 
$$d=(\ell-h)(\ell-h+1).$$
Since $d\ne 0$ by primitivity we have that $\ell \ne h$.
Therefore, we obtain that $k=2\ell(\ell-h+1)$ and $n = 4\ell(\ell+1)+1=(2\ell+1)^2$ with $\ell \not \in \{\pm h, -(h+1), h-1\}$, as it was to be shown. 

By taking into account that $0<r=\frac{1}{2}(e-d+\sqrt{\alpha})=\ell+h$ with multiplicity $m_r=2\ell(\ell-h+1) $ and $-1\ge s=\frac{1}{2}(e-d-\sqrt{\alpha})=h-\ell-1$ with $m_s=2\ell(\ell+h+1)$,
the energy of $\G$ is given by
$$E(\G) = k + r m_r - s m_s = 4\ell(\ell-h+1)(\ell+h+1).$$
	
Finally, the converse of all the items can be proved by performing direct calculations using Proposition~\ref{srg equien}.  

The last assertion is clear from the expressions of the energies in ($b$) and ($c$). In fact, in case ($b$) the integers $\ell-h$ and $\ell+h+1$ have different parity, so exactly one of them is even.
\end{proof}

We have the following direct consequence.
\begin{coro}
If the energy of a strongly regular graph is odd then it is not equienergetic with its complement.
\end{coro}

\begin{proof}
Let $\G$ be a strongly regular graph and suppose that $E(\G)$ is odd. If $\G$ is a conference graph, then $E(\G)=2d(1+\sqrt{4d+1})$ by ($a$) in Theorem \ref{Teo equien srg e le d}. Either if $\sqrt{4d+1}$ is an integer or not, $E(\G)$ is not an odd integer. If $\G$ is not a conference graph then $E(\G) \in 2\Z$ by ($b$) and ($c$) of Theorem \ref{Teo equien srg e le d}. In all the cases, the graph $\G$ cannot be equienergetic with $\bar \G$. 
\end{proof}

%\begin{rem}
%Let $\G$ be a strongly regular graph which is not a conference graph. By items ($b$) and ($c$) in Theorem \ref{Teo equien srg e le d}, if $\G$ is equienergetic with its complement then $4\mid E(\G)$. In fact, in case ($b$) the integers $\ell-h$ and $\ell+h+1$ have different parity, so exactly one of them is even. In this case $8 \nmid E(\G)$. However, in case ($c$) there can be more powers of 2 dividing the energy. This happens if $h$ and $\ell$ are not both odd integers.
%\end{rem}

\begin{rem}\label{rem cite}
In \cite{RPPA}, the authors found the following two families of strongly regular graphs which are equienergetic with their complements, by direct calculation. 
\begin{enumerate}[($a$)]
	\item $\G_1(t) = srg(4t^2,2t^2-t,t^2-t,t^2-t)$ with $t>1$. \msk 
	
	\item $\G_2(t) = srg(t^2, 3(t-1), t,6)$ with $t>2$.
\end{enumerate} 
It is easy to verify that these two families satisfy condition \eqref{cond Equien srg} in Proposition \ref{srg equien}.

\noindent ($i$) The graphs $\G_1(t)$ above 
can be obtained from $(b)$ of Theorem \ref{Teo equien srg e le d} by taking $h=0$. 
This theorem shows that the parameters of $\G_1(t)$ are the only possible ones for a strongly regular graph equienergetic with its complement and having $e=d$. The parameters of 
$\overline{\G_1(t)}$ are exactly the ones obtained by taking $h=1$ in $(b)$ of Theorem~\ref{Teo equien srg e le d}. \sk 

\noindent ($ii$)
The graphs $\G_2(t)$ 
can also be obtained either from ($b$) or ($c$) in Theorem \ref{Teo equien srg e le d} by taking $d=6$ and $t=2\ell$ or $d=6$ and $t=2\ell+1$, respectively (both cases with $h=\ell-2$).
\end{rem}

\section{Some families of strongly regular graphs}
Here we apply the results in the previous section to find particular pairs of strongly regular graphs equienergetic with their complements in some distinguished families.

\subsection{SRGs with fixed integral minimum eigenvalue}
The graph $aK_m$ has minimal eigenvalue $s=-1$, the pentagon $C_5$ has minimal eigenvalue $s=-\tfrac{1+\sqrt 5}2$, and all other strongly regular graphs have $s\le -2$.

\subsubsection{The case $s=-2$} 
We now characterize all SRGs with minimum eigenvalue $s=-2$. 

\begin{prop} \label{s=-2}
The only strongly regular graphs with minimal eigenvalue $s=-2$ which is equienergetic with its complement is the 4-cycle $C_4$ in the imprimitive case and the lattice graphs $L_2(n)$ with $n\ge 3$ and the Shrikhande graph $\G_{Shr}$ in the primitive case.  
\end{prop}

\begin{proof}
Seidel (\cite{Sei}) classified all strongly regular graphs with minimal eigenvalue $s=-2$. They are given by the following 3 infinite families and 7 isolated graphs:
\begin{enumerate}[($a$)]
	\item $K_{n\times 2}$ with parameters $srg(2n, 2n-2, 2n-4, 2n-2)$, $n\ge 2$, \sk 
	
	\item $L_2(n)$ with parameters $srg(n^2, 2n-2, n-2, 2)$, $n\ge 3$,  \sk 
	
	\item $T(n)$ with parameters $srg(\tfrac 12 n(n-1), 2n-4, n-2, 4)$, $n\ge 5$,  \sk 

	\item the Petersen graph with parameters $srg(10,3,0,1)$, \sk 

	\item the Clebsch graph with parameters $srg(16,10,6,6)$, \sk
	 
	\item the Shrikhande graph with parameters $srg(16,6,2,2)$, \sk
	 
	\item the Schl\"afli graph with parameters $srg(27,16,10,8)$, \sk
	 
	\item the 3 Chang graphs with parameters $srg(28,12,6,4)$. 
\end{enumerate}
Here, $K_{n \times 2}$ is the complete multipartite graph (imprimitive), $L_2(n)$ is the lattice graph --or the Hamming graph 
$H(2,n)$-- and $T(n)$ is the triangular graph. 

The case ($a$) is the only family of imprimitive graphs. By Proposition \ref{prop imprimitive}, $K_{m\times 2}$ is equienergetic with its complement if and only if $m=2$. In this case we have 
$$K_{2\times 2}=C_4=srg(4,2,0,2)= \overline{Cr(2)}.$$

The rest of the graphs are primitive and the result follows by applying Proposition~\ref{cond Equien srg}. 
For the lattice graphs $L_2(n)$ in ($b$) we have $\alpha = (n-4)^2 + 4(2n-4) = n^2$. In this case condition \eqref{srg equien} holds and hence $E(L_2(n)) = E(\overline{L_2(n)})$. On the other hand, for the triangular graphs $T(n)$ in ($c$) we have $\alpha = (n-2)^2$ and \eqref{srg equien} takes the form
$$ n^2-n = 2(n^2-3n+2)+2$$
which has roots $n=2,3$. But $n\ge 5$ by hypothesis, hence $T(n)$ is not complementary equienergetic. 

It is straightforward to check that the graphs ($d$), ($e$), ($g$) and ($h$) do not satisfy \eqref{srg equien}. 
Finally, observe that the Shrikhande graph $\G_{Shr}$ has the same parameters as $L_2(4)$, and thus, $E(\G_{Shr})=E(\overline{ \G_{Shr}})$. This concludes the proof. 
\end{proof}

Note that the pair $\{ C_4=K_{2,2}, 2K_2 = K_2 \otimes K_2 \}$ of imprimitive strongly regular graphs which are complementary equienergetic was previously obtained in Theorem \ref{Teo Bip equien} as a crown graph and its complement.

\begin{rem}
($i$) The Shrikhande graph $\G_{Shr}$ can be constructed as the Cayley graph $X(\Z_4 \times \Z_4, S)$ with connection set 
$S=\{ \pm ( 1 , 0 ), \pm ( 0 , 1 ), \pm ( 1 , 1 ) \}$. 
It has spectrum $\{[6]^1, [2]^6, [-2]^9 \}$ and hence $E(\G_{Shr})=36$.

($ii$) The Shrikhande graph have the same parameters as the $4 \times 4$ rook's graph $R_4$ --that can be seen as the line graph $L(K_{4,4})$-- and they are the only two graphs with parameters $srg(16,6,2,2)$. 
This implies that the rook's graph $R_4$ is also equienergetic with its complement.
However, $\G_{Shr}$ and $R_4$ are isospectral. 

($iii$) Notice that the parameters of the graphs $\G_{Shr}$ and $R_4$ correspond to the graph in ($b$) of Theorem \ref{Teo equien srg e le d} with $h=0$ and $\ell=2$, that is they have the parameters of $\G_{0,2}$.
\end{rem}

\subsubsection{The case $s=-m$, $m\ge 2$} 
Generalizing the characterization of all strongly regular graphs with minimum eigenvalue $s=-2$, Sims (\cite{RCh}, see also Neumaier \cite{Neu}) classified all strongly regular graphs $srg(n,k,e,d)$ with integral minimum eigenvalue $s=-m$, $m\ge 2$. They are of four types, 3 infinite families (complete multipartite graphs, Latin square graphs and Steiner graphs) and a finite number of sporadic graphs (not explicitly given) such that $1\le d\le m^3(2m-3)$. 

We next give the characterization of complementary equienergetic graphs of this kind.
\begin{prop} \label{s=-m}
Let $\G$ be a strongly regular graph with minimal eigenvalue $s=-m$, $m\ge 2$. If $\G$ is not sporadic, then $\G$ is equienergetic with its complement if and only if $\G$ is a complete multipartite graph $K_{m \times m}$ in the imprimitive case or $\G$ is a Latin square graph $LS_m(n)$ for any $n \ge 2$ in the primitive case.
\end{prop}

\begin{proof}
By the already mentioned result of Sims, $\G$ is either a complete multipartite graph $K_{n\times m}$ of parts of size $m$, a Latin square graph $LS_m(n)$ or the block graph $S_m(n)$ of a Steiner system $S(2,m, mn+m-n)$. 

We know by Proposition \ref{prop imprimitive} that $K_{n\times m}$ is equienergetic with its complement if and only if $n=m$. 
The Latin square graph $LS_m(n)$ has parameters 
$$srg \big( n^2, m(n-1), (m-1)(m-2)+n-2, m(m-1) \big).$$    
It is straightforward to check that $\alpha = n^2$ and \eqref{cond Equien srg} holds. Thus, by Proposition \ref{srg equien}, $LS_m(n)$ is equienergetic with $\overline{LS_m(n)}$. 

Finally, suppose that $\G$ is $S_m(n)$, the block graph of $S(2,m, mn+m-n)$. The block graph of $S(2,m,t)$ has parameters
$$srg(\tfrac{t(t-1)}{m(m-1)}, \tfrac{m(t-m)}{m-1}, (m-1)^2 + \tfrac{t-1}{m-1} -2, m^2).$$ 
Hence, $\G=S_m(n)$ has parameters 
$$srg \big( \tfrac{(mn+m-n)(mn+m-n-1)}{m(m-1)}, mn, (m-1)^2+n-1, m^2 \big).$$
Note that $S_2(n) = T(2+n)$, so by ($c$) in the proof of Proposition \ref{s=-2} this graph is not equienergetic with its complement. In general, after some tedious calculations one get that \eqref{cond Equien srg} does not hold and hence, by Proposition~\ref{srg equien}, 
$\G$ is not equienergetic with its complement in this case. 
This completes the proof. 
\end{proof}

\begin{rem}
($i$) Note that if we take $m=-2$ we recover Proposition \ref{s=-2}, since we have $K_{2\times 2}=C_4$ and $LS_2(n)=L_2(n)$. In this case the Shrikhande graph is obtained, for it is a strongly regular graph with the parameters of $L_2(4)$. 
Moreover, if $m=-3$ the graph $LS_3(n)$ for $n\ge 2$ has parameters $srg=(n^2,3(n-1),n,6)$ which are the graphs obtained by Ramane et al in \cite{RPPA}. It remains to study the sporadic cases with $s=-m$ and $m\ge 3$.

\noindent ($ii$)	
In \cite{NPS}, there are three classifications of strongly regular graphs: ($a$) with $s=3$ having feasible parameters of a block graph of a quasi-symmetric design --or QSD-- (Theorem 4), ($b$) strongly regular graphs with second eigenvalue $r=2$ (Theorem 5) and ($c$) strongly regular graphs with $r=2$ having feasible parameters of the block graph of a QSD (Theorem 6). 
One can also use Proposition \ref{srg equien} and/or Theorem \ref{Teo equien srg e le d} to classify all complementary equienergetic strongly regular graphs in these 3 families. 
\end{rem}

\subsection{Triangle-free SRGs}
The girth of a strongly regular graph $\G=srg(n,k,e,d)$ can only take the values 3, 4 and 5. If $e\ne0$ and $d\ne 0$ then $\G$ has triangles. Strongly regular graphs with parameters 
$$srg(n,k,0,d)$$ are triangle-free. When $d=1$ these graphs are Moore graphs and have girth 5. More generally, strongly regular graphs with parameters $srg(n,k,e,1)$ are geodetic graphs, i.e.\@ a graph in which every two vertices have a unique unweighted shortest path. The only known geodetic strongly regular graphs are the Moore graphs.

\begin{prop} \label{Tf srgs}
The only strongly regular Moore graph which is equienergetic with its complement is the 5-cycle $C_5=srg(5,2,0,1)$, which is self-complementary. Moreover, it is the only such graph among the known strongly regular graphs without triangles.   
\end{prop}

\begin{proof}
Hoffman and Singleton (see \cite{HS}) characterized the pairs $(n,k)$ such that there is some strongly regular graph with parameters $(n,k,0,1)$, i.e.\@ with girth 5, showing that  
$$(n,k) \in \{(5,2),(10,3),(50,7),(3250,57)\}.$$
The 5-cycle or pentagon $C_5$, the Petersen graph and the Hoffman-Singleton graph are the only graphs with the first 3 parameters, respectively. It is not known if an $srg(3250, 57,0,1)$ exists. It is straightforward to check that \eqref{cond Equien srg} holds for $C_5$ and do not hold for the other parameters. However, $C_5 = \overline C_5$ and hence $E(C_5)= E(\overline C_5)$ trivially. Note that $C_5$ is the conference graph $\Gamma_1$ in ($a$) of Theorem \ref{Teo equien srg e le d}.

Apart from the complete graphs $K_1$ and $K_2$ (with $g=\infty$) and the complete bipartite graphs $K_{n,n}=srg(n^2,n-1,0,n)$ (with $g=4$), there are only seven known graphs with $g \ge  4$. These are the pentagon, Petersen and Hoffman-Singleton graphs mentioned above and the 5-folded cube, the Gewirtz graph, the Mesner $M_{22}$ graph and the Higman-Sims graphs with parameters $$srg(16,5,0,2), \quad srg(56,10,0,2), \quad srg(77,16,0,4) \quad \text{and} \quad (100,22,0,6),$$ 
respectively. One can check that non of these graphs satisfy \eqref{cond Equien srg}, and hence they are not equienergetic with their complements. 
\end{proof}

\subsection{Semiprimitive GP-graphs}
A generalized Paley graph (GP-graph) is the Cayley graph $\G(k,q)=X(\ff_q,R_k)$ where $\ff_q$ is a finite field of $q$ elements, with $q=p^m$ for some prime $p$ and integer $m$, and 
$$R_k = \{ x^k : x\in \ff_q^*\}.$$ 
The complement of $\G(k,q)$ is also a Cayley graph, $\overline{\Gamma(k,q)} = X(\ff_{q}, (R_{k})^c \smallsetminus \{0\})$. 

A generalized Paley graph $\Gamma(k,q)$ is called \textit{semiprimitive} if 
$k=2$ and $q \equiv 1 \pmod 4$ or else if 
$$k>2, \quad m \text{ is even,} \quad 
\text{$k \mid p^t+1$ for some $t\mid \tfrac m2$} \quad \text{ and } \quad k \ne p^{\frac m2}+1.$$ 
Semiprimitive GP-graphs were defined in \cite{PV3}, since they are related with semiprimitive cyclic codes,
where some of their spectral and structural properties are studied.

\begin{prop} \label{eqnoisocomp}
Let $\Gamma(k,q)$ be a semiprimitive GP-graph with $k>2$. Then, the graphs $\Gamma(k,q)$ and $\overline \Gamma (k,q)$ 
are equienergetic and non-isospectral if and only if
$s=\frac{m}{2t}$ is odd, where $t$ is the least integer $j$ such that $k \mid p^j+1$. 
\end{prop}

\begin{proof}
In \cite{PV3}, Theorem 3.3, we show that $\G=\G(k,q)$, with $k>2$ and $q=p^m$ with $m$ even is a connected strongly regular graph, with integral spectrum given by 
$$Spec(\G) = \{ [\tfrac{q-1}k]^1, [\lambda_1]^{\frac{q-1}k}, [\lambda_2]^{(k-1)\frac{q-1}k}\}$$ 
where 
$$\lambda_1 = (-1)^{s+1} \tfrac{ (k-1) \sqrt q-1}{k} \qquad \text{and} \qquad \lambda_2 = -\tfrac{(-1)^{s+1} \sqrt q +1}{k},$$	
with $s$ as defined in the statement. By Proposition \ref{lem delta} and ($a$) in Lemma \ref{lemin} we have 
$$m(0) + \sigma(\G) = \Delta(\G) = \tfrac{2(q-1)}k +1 -q = \tfrac{(2-k)(q-1)}k .$$ 
Since $\lambda_1, \lambda_2 \ne 0$ we have that $m(0)=0$. Also,  $\sigma(\G) = \pm (m_{\lambda_1} -m_{\lambda_2}$) depending the case, hence 
$$\sigma(\G) = (-1)^{s+1} (\tfrac{q-1}k - (k-1)\tfrac{q-1}k) =(-1)^{s+1} \tfrac{(2-k)(q-1)}{k}.$$
Hence, $\G(k,q)$ and $\overline{\G(k,q)}$ are equienergetic if and only if $(-1)^{s+1} =1$. That is to say, if and only if $s$ is odd.
\end{proof}
In Proposition 3.5 of \cite{PV3} we gave the parameters $srg(q,\tfrac{q-1}k,e,d)$ of $\G(k,q)$, so one can prove the previous result also by using Proposition \ref{srg equien}. 

\begin{exam}
Consider the pair $(3,p^{2t})$ where $p$ is a prime with $p \equiv 2 \pmod 3$ and $t\ge 1$.
It is easy to see that $(3,p^{2t})$ in the above conditions is semiprimitive.
In this case, one can check that the sign is $(-1)^{t+1}$, i.e.\@ $s=t$.
By Proposition \ref{eqnoisocomp}, we obtain the following infinite family 
of semiprimitive GP-graphs equienergetic with their own complements
$$\{ \G(3,2^{4n+2})\}_{n\in\mathbb{N}} \cup \{ \G(3,p^{4n-2})\}_{n\in\mathbb{N}}$$ 
where $p$ runs over all odd primes $p\equiv 2 \pmod{3}$. %, p\ne 2.
\end{exam}

\subsection{Spectrally determined SRG's}
A graph which is uniquely determined by its spectrum is called a \textit{DS graph}. 
Strongly regular DS graphs are classified (see for instance \S14.5 in \cite{BH}). 

Let us see that DS strongly regular graphs which are equienergetic with its complements are imprimitive or Paley graphs  (conference, self-complementary). 
More precisely we have the following. 

\begin{prop} \label{DS srg}
If $\G$ is a connected DS strongly regular graph, then $\G$ is complementary equienergetic if and only if 
$\G$ is either the complete multipartite graph $K_{m\times m}$ with $m\ge 2$, or one of the Paley graphs $P(5)$, $P(13)$, $P(17)$, or the lattice graph $L_2(n) = L(K_{n,n})$ for all $n\ge 5$. 
\end{prop}

\begin{proof}
The DS strongly regular graphs are given by 3 infinite families and some sporadic graphs. The infinite families are 
(see Proposition 14.5.1 in \cite{BH}): 
$$ aK_\ell, \qquad L(K_n) \:\: \text{with $n\ne 8$,} \qquad and \qquad L(K_{m,m}) \:\: \text{with $m\ne 4$}.$$

The graph $aK_\ell$ is imprimitive and we know by Proposition \ref{prop imprimitive} that it is equienergetic with its complement if and only if $a=\ell=m$ for $m\ge 2$. The complement of $mK_m$ is $K_{m\times m}$, which is connected. The graph $L(K_n)$ is the triangular graph $T(n)$ and $L(K_{m,m})$ is the lattice graph $L_2(m)$. They have minimum eigenvalue $s=-2$ and hence they have been treated in the proof of Proposition \ref{s=-2}. There we showed that $L_2(n)=L(K_{n,n})$ is complementary equienergetic for all $n\ge 3$ while $T(n)$ is not.

The sporadic graphs are given in the following tables taken from \cite{BH} (see Table 14.2 in \cite{BH} for details). 
\renewcommand{\arraystretch}{1.1}

\begin{table}[h!]
\caption{Sporadic DS strongly regular graphs of conference type}
\begin{tabular}{|c|c|c|}
\hline
$n$ & spectrum & name \\ \hline 
5 & 2, $(\tfrac{-1 \pm \sqrt 5}{2})^2$ & Paley $P(5)$ \\
13 & 6, $(\tfrac{-1 \pm \sqrt 13}{2})^2$ & Paley $P(13)$ \\
17 & 8, $(\tfrac{-1 \pm \sqrt 17}{2})^2$ & Paley $P(17)$ \\ 
\hline 
\end{tabular}
\end{table}

\begin{table}[H]
	\caption{Sporadic DS strongly regular graphs of non-conference type}
\begin{tabular}{|c|c|c|}
\hline
$n$ & spectrum & name \\ \hline 
16 & 5, $1^{10}$, $(-3)^5$ & folded 5-cube \\
27 & 10, $1^{20}$, $(-5)^6$ & GQ(2,4) \\
50 & 7, $2^{28}$, $(-3)^{21}$ & Hoffman-Singleton \\
56 & 10, $2^{35}$, $(-4)^{20}$ & Gewirtz \\
77 & 16, $2^{55}$, $(-6)^{21}$ & Mesner $M_{22}$ \\
81 & 20, $2^{60}$, $(-7)^{20}$ & Brouwer-Haemers \\
100 & 22, $2^{77}$, $(-8)^{22}$ & Higman-Sims \\
105 & 32, $2^{84}$, $(-10)^{20}$ & flags of PG(2,4) \\
112 & 30, $2^{90}$, $(-10)^{21}$ & GQ(3,9) \\
120 & 42, $2^{99}$, $(-12)^{20}$ & 001.. in $S(5,8,24)$ \\
126 & 50, $2^{105}$, $(-13)^{20}$ & Goethals \\
162 & 56, $2^{140}$, $(-16)^{21}$ & local McLaughlin \\
176 & 70, $2^{154}$, $(-18)^{21}$ & 01.. in $S(5,8,24)$ \\
275 & 112, $2^{252}$, $(-28)^{22}$ & McLaughlin \\
\hline
\end{tabular}
\end{table}

For any $q\equiv 1 \pmod 4$, the Paley graph $P(q)$ is defined as the Cayley graph $Cay(\ff_q,S)$ where $S=\{x^2: x\in \ff_q^*\}$.
It is known that Paley graphs are strongly regular graph with parameters 
	$$P(q)=srg(q, \tfrac{q-1}2, \tfrac{q-5}4, \tfrac{q-1}4).$$
Paley graphs are conference graphs, since they satisfy \eqref{conf} with $q=4t+1$, and hence self-complementary; so they are trivially equienergetic with their complements. 
For the non-conference graphs, it is clear that 
	$$m_r -m_s >0 \qquad \text{and} \qquad 2k+1-n<0$$ 
for all the graphs in the table. 
By \eqref{delta srg} we have that $\Delta(\G)=m_r-m_s$ and hence, by Proposition \ref{lem delta}, 
none of these graphs can be equienergetic with its complement. This completes the proof.
\end{proof}

\section{Orthogonal arrays and Latin square graphs}
Orthogonal arrays are generalizations of orthogonal Latin squares. Let $n,m$ be non-negative integers.
An \textit{orthogonal array} (of index 1 and strength 2), denoted by $OA(n,m)$, is an $m \times n^2$ matrix $A$ with entries in a set of cardinality $m\ge 2$, if each set of two rows ($2\times n^2$ sub-matrix) contains all $2\times 1$ column vectors exactly once. 
An example of an $OA(2,3)$ is as follows
$$
{\footnotesize \begin{pmatrix}
1 & 2 & 1 & 2\\
1 & 2 & 2 & 1\\
1 & 1 & 2 & 2
\end{pmatrix}}.$$
An orthogonal array $OA(n,m)$ is equivalent to a transversal design $TD(m,n)$ or to a set of $m-2$ mutually orthogonal Latin squares (MOLS) of order $n$.

Given an $OA(n,m)$, there is an associated graph $\G$. The vertices of $\G$
are the $n^2$ columns of the orthogonal array, 
and two vertices are adjacent if they have the same entries in exactly one coordinate position. 
The graph associated to the previous {\blue $OA(2,3)$} is $K_4$.
It is well-known that the graph defined by an orthogonal array $OA(n,m)$ is a strongly regular graph with parameters
\begin{equation} \label{OA pars}
srg \big(n^2,m(n-1), m^2-3m+n, m(m-1) \big).
\end{equation}
If $\G$ is a strongly regular graph with the same parameters as above, it is said that it has \textit{$OA(n,m)$ parameters} or that it is a \textit{pseudo Latin square} graph $PL(n,m)$. Notice that we have the following two special cases:
$$PL(n,m)=
\begin{cases}
K_{n^2} & \quad \text{if $m=n+1$},\\[1mm]
K_{n\times n} & \quad \text{if $m=n$}.\\
\end{cases}$$

We now show that a graph having $OA(n,m)$ parameters with $m \notin \{%\frac{n+1}2, 
n,n+1\}$ is complementary equienergetic.

\begin{prop}\label{OA equien}
	Let $\G$ be a strongly regular graph. 
	If $\G$ has $OA(n,m)$ parameters with $m\notin \{ n, n+1\}$, then $\G$ and $\overline \G$ are equienergetic. 
	Furthermore, if $m\ne \frac{n+1}{2}$ then $\G$ and $\overline \G$ are non-isospectral.
\end{prop}

\begin{proof}
First note that if $\G$ has $OA(n,m)$ parameters, then $\Gamma$ is clearly primitive. If $\G=srg(n,k,d,e)$ then 
$\overline \G=srg(n,\bar k,\bar e, \bar d)$ and we have 
	$$\bar d = n^2-2m(n-1) + m^2-3m+n =(n-m)(n-m+1) \ne 0$$
since $m\neq n$ and $m\neq n+1$, by hypothesis. 
One can check that by \eqref{OA pars} condition \eqref{cond Equien srg} holds and thus the result follows directly from Proposition \ref{srg equien}. 
On the other hand, condition $m\ne \frac{n+1}{2}$ implies that $\G$ and $\overline \G$ have different regularity degrees and so these graphs are non-isospectral, as asserted.
\end{proof}

Notice that both families of graphs appearing in Remark \ref{rem cite} have the parameters of an $OA(n,m)$ for some $n,m$. 
More precisely, the graphs $\G_1(t)$ has $OA(2t,t)$ parameters and the graphs $\G_2(t)$ has $OA(t,3)$ parameters.
The following result shows that this happens with any strongly regular graph equienergetic with its complement which is not a conference graph.

\begin{prop}\label{main thm}
	Let $\G$ be a primitive strongly regular graph.
	Then, $\G$ and $\overline \G $ are equienergetic if and only if 
	$\G$ is a conference graph or $\G$ has $OA(n,m)$ parameters for some $n,m$ with $m\not \in\{n,n+1\}$. 
	In the non-conference case, the graphs $\G$ and $\overline \G$ are also non-isospectral.
\end{prop}

\begin{proof}
Since conference graphs are self-complementary, $\G$ is trivially equienergetic with its complement in this case.
On the other hand, if $\G$ has $OA(n,m)$ parameters with $m\not \in\{n,n+1\}$, 
then $\G$ and $\overline \G$ are equienergetic by Proposition \ref{OA equien}.
	
Now, if $\G$ and $\overline \G$ are equienergetic, by Theorem \ref{Teo equien srg e le d}
we have that $\G$ is either a conference graph or it is one of the graphs in cases ($b$) or ($c$) in that theorem, depending 
on whether $e-d$ is even or odd different from $-1$. 
Notice that parameters of $\G$ corresponds in case ($b$) and ($c$) respectively to the parameters of 
	$$OA(2\ell, \ell-h) \qquad \text{and} \qquad OA(2\ell+1, \ell-h+1),$$ 
as desired.
Finally, one can check that in these two cases, the graphs $\G$ and $\overline \G$ has different degrees of regularity and hence they are non-isospectral.
\end{proof}

Putting together Propositions \ref{prop imprimitive} and \ref{main thm}, 
we obtain the classification of all connected strongly regular graphs
equienergetic and non-isospectral with their own complements.

\begin{thm} \label{Equien noisosp prim. car}
	Let $\G$ be a connected strongly regular graph.
	Then, $\G$ and $\overline \G $ are equienergetic non-isospectral graphs if and only if $\G$ has $OA(m,m)$ parameters (in the imprimitive case) or else $\G$ has $OA(n,m)$ parameters for some integers $n\ne m$ with $m\ne n+1,\frac{n+1}{2}$ (in the primitive case).
\end{thm}

\begin{proof}
If $\G$ is imprimitive, the result is implied by Proposition \ref{prop imprimitive} by noting that the graph 
$K_{m \times m} = srg(m^2, (m-1)m, (m-2)m, (m-1)m)$ has $OA(m,m)$ parameters taking $n=m$ in \eqref{OA pars}.	
On the other hand, if $\G$ is primitive, the result follows by Proposition~\ref{main thm}.
\end{proof}

\begin{rem} 
By the theorem, all the graphs in the previous section which are equienergetic and non-isospectral with their complements, which are not conference graphs (as $C(5)=P(5)$, $P(13)$ and $P(17)$), have the parameters of some $OA(n,m)$. 

\noindent ($i$) The graphs in Propositions \ref{s=-2}, \ref{s=-m} and \ref{DS srg} %which are equienergetic with their complement 
satisfy:

	\begin{enumerate}[($a$)]
		\item the cycle $C_4$ has the parameters of $OA(2,2)$, \sk 
		
		\item the Shrikhande graph has the same parameters as $OA(4,2)$, \sk 
		
		\item the lattice graph $L_{2}(n)$ has the parameters of $OA(n,2)$ with $n\ge 3$, \sk 
		
		\item the graph $K_{m\times m}$ has the same parameters as $OA(m,m)$, \sk 
		
		\item the graph $LS_{m}(n)$ has the parameters of $OA(n,m)$ with $m\ne n+1,\frac{n+1}2$.
	\end{enumerate} 

\noindent ($ii$) The strongly regular graphs found in \cite{RPPA} and given in Remark \ref{rem cite} have OA parameters. In fact $\G_1(t)$ is an $OA(2t,t)$ for $t>1$ and $\G_2(t)$ is an $OA(t,3)$ for $t>2$.

\noindent ($iii$) The semiprimitive GP-graphs $\G(k,q)$ with $s$ odd of Proposition \ref{eqnoisocomp}, with $q=p^{2m}$ and $s=\frac{m}{2t}$ where $t$ is the least integer $j$ such that $k \mid p^j+1$, has the parameters of 
$$OA(\sqrt q,\tfrac{\sqrt q+1}k)=OA(p^m,\tfrac{p^m+1}k).$$
\end{rem}

\subsection*{Cameron's hierarchy}
There is a hierarchy of regularity conditions on graphs due to Cameron (\cite{Ca}). For a non-negative integer $t$ and sets $S_1$ and $S_2$ of at most $t$ vertices, let $\mathcal{C}(t)$ be this %the following 
graph property:
if the induced subgraphs on $S_1$ and $S_2$ are isomorphic, then the number of vertices joined to every vertex in $S_1$ is equal to the number joined to every vertex in $S_2$. 
A graph satisfying property $\mathcal{C}(t)$ is called \textit{$t$-tuple regular}. 

Conditions $\mathcal{C}(t)$ are stronger as $t$ increases. A graph satisfies $\mathcal{C}(1)$ if and only if it is regular and it satisfies $\mathcal{C}(2)$ if and only if it is strongly regular. If a graph satisfies $\mathcal{C}(3)$ then it is the pentagon $C_5$ or it has the parameters of a pseudo Latin square, a negative Latin square or a Smith type graph (see \cite{Ca2}). Up to complements, there are only two known examples of graphs satisfying $\mathcal{C}(4)$ but not $\mathcal{C}(5)$, the Schl\"afli graph and the McLaughlin graph.  Finally, the hierarchy is finite; if a graph satisfies $\mathcal{C}(5)$ then it satisfies $\mathcal{C}(t)$ for any $t$. The only such graphs are $aK_m$ and its complement for $a,m \ge 1$, the pentagon $C_5$ and the $3 \times 3$ square lattice $L(K_{3,3})$ (see \cite{Ca3}).

We have the following result relative to complementary equienergeticity on graphs satisfying property $\mathcal{C}(t)$.

\begin{thm} \label{teo Ct}
	Let $\G$ be a connected strongly regular graph satisfying $\mathcal{C}(3)$ or $\mathcal{C}(5)$. Then $\G$ is equienergetic with $\overline \G$ if and only if $\G$ is one of the following: 
	\begin{enumerate}[$(a)$]
		\item $K_{m\times m}$ for any $m\ge 2$ (imprimitive case), \sk 
		
		\item the pentagon $C_5$ (conference, primitive case), or \sk 
		
		\item $PL_n(m)$ for any $n,m \ge 1$ (non-conference, primitive case).
	\end{enumerate}
\end{thm}

\begin{proof}
Suppose first that $\G$ satisfies $\mathcal{C}(3)$. Then, by the comments previous to the statement, 
$\G$ is the pentagon $C_5$, a pseudo Latin square graph $PL_n(m)$, a negative Latin square graph $NL_n(m)$ or a Smith graph. The graph $C_5$ is self-complementary so trivially complementary equienergetic. Also, $PL_n(m)=OA(n,m)$. The negative Latin square graphs are strongly regular graphs taking negative parameters in $LS_n(m)$. Hence, we have
$$NL_n(m)=LS_{-n}(-m) = srg \big(n^2, m(n+1), m^2+3m-n, m(m+1) \big).$$ 
Thus, negative Latin square graphs do not have $OA$ parameters, and hence cannot be complementary equienergetic.

The Smith graphs have parameters $srg(v,k,e,d)$ with 
\begin{equation} \label{smith}
\begin{aligned}
&v = \frac{2(r-s)^2\{(2r+1)(r-s)-3r(r+1) \}}{(r-s)^2 - r^2(r+1)^2}, \\[2mm] 
&k = \frac{-s\{(2r+1)(r-s)-r(r+1) \}}{(r-s) + r(r+1)}, \\[2mm] 
&e = \frac{-r(s+1) \{ (r-s)-r(r+3) \}}{(r-s) + r(r+1)}, \\[2mm]
&d = \frac{-s(r+1) \{ (r-s)-r(r+1) \}}{(r-s) + r(r+1)},
\end{aligned}
\end{equation}
where $r$ and $s$ are the eigenvalues distinct from $k$.
Since they are primitive non-conference graphs it is enough to show that they do not have $OA$ parameters.
Assume the contrary holds, i.e.\@ that they have $OA(n,m)$ parameters, and recall that for these graphs we have 
{\blue that} $r>0$ and $s<-1$ are both integers.
Thus, we have $s=-m$ and $d=m(m-1)=s(s+1)$. By \eqref{smith} we have  
$$s(s+1)=\frac{-s(r+1) \{ (r-s)-r(r+1) \}}{(r-s) + r(r+1)}.$$
Since $s \ne 0$ we can cancel it, so that we have  
$$(s+1)(r-s) + r(s+1)(r+1)=-(r+1)(r-s)+r(r+1)^2$$
from where we arrive at
$$(r+s+2)(r-s)=r(r+1)(r-s).$$
Canceling $r-s>0$ we finally get $s+2=r^2$, which cannot occur since $r>0$ and $s\le -2$.

Now, suppose that $\G$ satisfies $\mathcal{C}(5)$. By previous comments, $\G$ is either $K_{a\times m}$ or its complement for $a,m \ge 1$, or $C_5$ or else $L(K_{3,3})$. By Proposition \ref{prop imprimitive}, $K_{m \times m}$ (and its complement) is complementary equienergetic. Also, $C_5$ is self-complementary, hence trivially complementary equienergetic Finally, the graph $L(K_{3,3})$ was ruled out in the proof of Proposition \ref{DS srg}. 
This completes the proof.
\end{proof}

We have thus completely settled the question of complementary equienergeticity for bipartite graphs satisfying $\mathcal{C}(1)$ and for any graph satisfying $\mathcal{C}(2)$, $\mathcal{C}(3)$ and $\mathcal{C}(5)$.  
We also solved the same question for all known graphs satisfying $\mathcal{C}(4)$. 
So, it remains to complete the characterizations for non-bipartite regular graphs which are not strongly regular graphs (i.e.\@ satisfying $\mathcal{C}(1)$) and those satisfying $\mathcal{C}(4)$, whose classification is still missing.

We close the section with a question. 
\begin{quest}
Are there any complementary equienergetic strongly regular graphs, necessarily with $OA$ parameters, satisfying $\mathcal{C}(4)$?
\end{quest}

\section{Unitary Cayley graphs}
In Section 3 we have classified all bipartite regular graphs which are complementary equienergetic (crown graphs and $C_4$). In Sections 4--6 we study the problem for strongly regular graphs, characterizing strongly regular graphs which are complementary equienergetic and not isospectral as those having OA parameters (i.e., pseudo Latin square graphs). The problem for non-bipartite regular graphs is open and seems quite unmanageable in general. In this section we deal with a particular class of  regular graphs, which includes many non-bipartite graphs in it.

Let $R$ be a finite commutative ring and $R^*$ its group of units. Consider the Cayley graph $X(R,R^*)$  with vertex set $R$ and connection set the units $R^*$ of $R$. We denote this graph simply by $G_R$. The graph $G_R$ is $|R^*|$-regular and loopless since $R^*$ is symmetric.
By the well-known Artin's structure theorem we have  
\begin{equation} \label{artin desc}
R = R_1 \times \cdots \times R_s
\end{equation} 
where each $R_i$ is a local ring, that is having a unique maximal ideal $\frak{m}_i$, with $|\frak{m}_i| = m_i$ say.
Moreover, one also has the decomposition at the level of units, $R^* = R_1^* \times \cdots \times R_s^*$. 
This implies that 
\begin{equation} \label{artin units}
G_R = G_{R_1} \otimes \cdots \otimes G_{R_s} 
\end{equation} 
with $G_{R_i} = X(R_i,R_i^*)$, and where $\otimes$ denotes the Kronecker product of graphs.
In \cite{PV5}, Lemma~3.4, we proved that $G_R$ is non-bipartite if and only if $2m_i < |R_i|$ for every $i=1,\ldots,s$.

The spectrum of $G_R = X(R,R^*)$ is known.  
For each subset $C \subseteq \{ 1, \ldots, s\}$ put  
\begin{equation} \label{lambdaC} 
\lambda_{C} = (-1)^{|C|} \frac{|R^*|}{\prod\limits_{j \in C} (|R_j^*|/m_j)} .
\end{equation} 
Then, the eigenvalues of $G_R$ are (\cite{ABJK}, see also \cite{Ki+}) 
\begin{equation} \label{spec GR}
\lambda = \begin{cases}
\lambda_C, & \quad \text{repeated $\prod\limits_{j \in C} (|R_j^*|/m_j)$ times,} \\ 
0,		 & \quad \text{with multiplicity $|R|- \prod\limits_{i=1}^s (1+ \tfrac{|R_i^*|}{m_i})$,}
\end{cases}
\end{equation}
where $C$ runs over all the subsets of $\{1,2,\ldots,s\}$. 
Note that, a priori, different subsets $C$ can give the same eigenvalue. 

We will need the following notations.
If we denote $q_i=\frac{|R_i|}{m_i}$, then $|R_i^*|=m_i(q_i-1)$ for $i=1,\ldots,s$.
For $C \subseteq \{1,\ldots,s\}$, we define
\begin{equation}\label{Pc}
P_C :=\prod_{j\in C} \tfrac{|R_j^*|}{m_j} = \prod_{j\in C} (q_j-1).
\end{equation}
A simple combinatoric argument shows that
\begin{equation}
\prod_{i=1}^{s}q_i=\sum_{C\subseteq\{1,\ldots,s\}}P_C.
\end{equation}
We also define the following numbers
\begin{equation} \label{S even odd}
S_e = \sum_{\substack{ C\subseteq\{1,\ldots,s\} \\ 0<|C|<s, \, |C| \, \text{even} }} P_C \qquad \quad \text{and} \quad \qquad 
S_o = \sum_{\substack{ C\subseteq\{1,\ldots,s\} \\ |C|<s, \, |C| \, \text{odd} }} P_C.
\end{equation}

In the previous notations we have the following. 
\begin{thm}
	Let $R$ be a finite commutative ring with unity, having Artin decomposition $R=R_1 \times \cdots \times R_s$ with $s$ even. 
Then, $E(G_R)=E(\overline{G_R})$ if and only if $R=\ff_{q_1}\times \ff_{q_2}$ is a product of two finite fields. 
\end{thm}

\begin{proof}
	Suppose that $G_R$ is equienergetic with its complement. Then, by Proposition~\ref{lem delta} we have  
	$\Delta(\G)= 2k+1-n$ and thus, by \eqref{spec GR}, we get 
	\begin{equation}\label{Equien verif}
	\sum_{\substack{ C\subseteq\{1,\ldots,s\} \\ |C|>0 \text{ even}}} P_C- 
	\sum_{\substack{ C\subseteq\{1,\ldots,s\} \\ |C| \text{ odd}}} P_C +
	|R|-\prod_{i=1}^s(1+\tfrac{|R_{i}^*|}{m_i})=2|R^*|-|R|+1.
	\end{equation}
	
Since $s$ is even, we have  
	$$\sum_{\substack{ C\subseteq\{1,\ldots,s\} \\ |C|>0 \text{ even} }} P_C 
	= S_e + \prod_{i=1}^s \tfrac{|R_{i}^*|}{m_i}
	= S_e + \frac{r_1 \cdots r_s}{m_1 \cdots m_s}$$
	where $r_{i}=|R_{i}^*|$ and $S_e$ is as in \eqref{S even odd}. 
	Notice that 
	$$\prod_{i=1}^s (1+\tfrac{|R_{i}^*|}{m_i}) = \frac{|R|}{m_1 \cdots m_s}.$$ 
	On the other hand, since $R^*=R_{1}^*\times \cdots \times R_{s}^*$ then $|R^*|=r_1\cdots r_s$. So \eqref{Equien verif} takes the form
	$$S_e+  \tfrac{r_1 \cdots r_s}{m_1 \cdots m_s} -S_o +|R| - \tfrac{|R|}{m_1\cdots m_s}=2r_1\cdots r_s -|R|+1$$
with $S_o$ as in \eqref{S even odd}.
	Thus, we obtain that
	$$|R| \cdot  \tfrac{2m_1\cdots m_s -1}{m_1\cdots m_s} =
	r_1\cdots r_s \cdot \tfrac{2m_1\cdots m_s -1}{m_1\cdots m_s} + 1 + S_o - S_e.$$
	Using $r_{i}=m_{i}(q_i -1)$ for $i=1,\ldots,s$, since $q_{i}=\frac{|R_{i}|}{m_i}$, we have the expression
	$$(2m_1\cdots m_s -1) \big(q_1\cdots q_s- (q_1-1)\cdots (q_s-1) \big)= 1+S_o-S_e . $$
	By taking into account that 
	$$q_1\cdots q_s = \sum_{C\subseteq \{1,\ldots,s\}} P_C \qquad  \text{and} \qquad P_C=\prod_{i\in C} (q_i-1),$$ 
	we obtain that $q_1\cdots q_s- (q_1-1)\cdots (q_s-1)=1+S_e+S_o$ and thus
	$$(2m_1\cdots m_s -2)(1+S_o)+(2m_1\cdots m_s -1) S_e=-S_e.$$
	Notice that the left hand side of the equality is greater than or equal to zero, since $m_i\ge 1$ and $P_C\ge 1$, and 
that the right hand side is less than or equal to zero, so 
	$S_e=0$; but this can only happen if and only if $s=2$. Therefore, $s=2$ as desired. 
	
	Conversely, if $R=\ff_{q_1} \times \ff_{q_2}$, the graphs $G_R$ and $\overline{G_R}$ are equienergetic, by Theorem 4.1 in \cite{PV5}.
\end{proof}

\begin{rem}
($i$) If $R=\ff_q\times \ff_q$, then $G_R$ is a strongly regular graph with parameters $srg(q^2,(q-1)^2,(q-2)^2,(q-1)(q-2))$, by Proposition 3.9 in \cite{PV5} and thus,    
by \eqref{OA pars}, $G_R$ has $OA(q,q-1)$ parameters. However, in general if $R=\ff_q \times \ff_{q'}$ with $q \ne q'$, then $G_R$ is not a strongly regular graph. 

\noindent ($ii$)
Note that by the observation after \eqref{artin units}, $G_R$ is bipartite if $R=\ff_2 \times \ff_q$ and it is non-bipartite if 
$R=\ff_q \times \ff_{q'}$ with $q,q'\ge 3$.
\end{rem}

The case in which the number of local factors of $R$ is odd greater than 1 is more involved (the local case is known from \cite{PV5}) and we can only give a necessary and sufficient condition for complementary equienergeticity. In this case, the graphs will not be strongly regular in general.

\begin{prop}
	Let $R$ be a finite commutative ring with unity, having Artin decomposition $R=R_1\times\cdots\times R_s$ with $s$ odd. 
\begin{enumerate}[$(a)$]
	\item If $s=1$, then $E(G_R)=E(\overline{G_R})$ if and only if $|R|=|\frak m|^2$ where $\frak m$ is the unique maximal ideal of $R$.

	\item If $s\ge 3$, then $E(G_R)=E(\overline{G_R})$ if and only if
\begin{equation} \label{s odd cond}
M S_e  +  (M -1)  (1+ S_o) = (q_1-1) \cdots (q_s-1),
\end{equation}
where $M=m_1 \cdots m_s$, $S_e$ and $S_o$ are as in \eqref{S even odd}.
\end{enumerate}
\end{prop}
	
\begin{proof}
($a$) Since $s=1$, $R$ is a local ring with maximal ideal $\frak m$.	
We have showed in \cite{PV5}, Proposition 4.1, that $G_R$ and $\overline{G_R}$ are equienergetic if and only if $|R|=|\frak m|^2$.

($b$) Assume that $s\ge 3$ is odd. In this case, we have  
	$$\sum_{\substack{ C\subseteq\{1,\ldots,s\} \\ |C| \text{ odd}}} P_C = 
	S_o + \frac{r_1 \cdots r_s}{m_1 \cdots m_s}.$$
	By \eqref{Equien verif} we obtain
	$$ S_e-S_o-\tfrac{r_1 \cdots r_s}{m_1 \cdots m_s}+|R|- \tfrac{|R|}{m_1\cdots m_s}=2r_1\cdots r_s -|R|+1.$$
	In a similar manner as in the even case we get
	$$q_{1}\cdots q_{s}(2m_1\cdots m_s -1)=(q_1-1)\cdots (q_s-1)(2m_1\cdots m_s+1)+1+S_o-S_e.$$ 
	By taking into account that $q_1\cdots q_s=1+S_e+S_o+(q_1-1)\cdots(q_s-1)$ we have
	$$(1+S_e+S_o)(2m_1\cdots m_s-1)=2(q_1-1)\cdots (q_s-1)+1+S_o-S_e .$$
	So, we obtain that 
	$$(2m_1\cdots m_s-2)(1+S_o)+2m_1\cdots m_s S_e = 2(q_1-1)\cdots (q_s-1).$$
	By removing $2$ on both sides of the equality, we obtain \eqref{s odd cond}, and the proof is complete.
\end{proof}

\begin{rem}
It is well-known (\cite{ABJK}, \cite{Ki+}) that if $(R,\frak m)$ is local ring, then 
		$G_R$ is a complete multipartite graph with $|R|/|\frak m|$ parts of the same size $|\frak m|$. Therefore, 
		$G_R$ is an imprimitive strongly regular graph in this case, and non-bipartite by Proposition 3.2 in \cite{PV5}. 
		Moreover, if $|R|=|\frak m|^2$, then $G_R$ has $OA(m,m)$ parameters which is in coincidence with 
		Theorem~\ref{Equien noisosp prim. car}. \sk 
\end{rem}	
		
Note that in case ($b$) of the previous proposition,
if $R=\ff_{q_1} \times \cdots \times \ff_{q_s}$ is a product of finite fields with $s\ge 3$, then $m_i=1$ for all $i=1,\ldots, s$, and hence \eqref{s odd cond} simply reads
\begin{equation} \label{ff prod}
S_e=(q_1-1) \cdots (q_s-1).
\end{equation} 

\begin{rem}
When $s=3$ and $m_1=m_2=m_3=1$, i.e.\@ if $R=\ff_{q_1} \times \ff_{q_2} \times \ff_{q_3}$,
we have showed in \cite{PV5}, Theorem 4.1, that $G_R$ is equienergetic with $\overline{G_R}$ if and only if
	$$R=\ff_{3}\times \ff_{5}\times \ff_{5} \qquad\text{or} \qquad R= \ff_{4}\times \ff_{4}\times\ff_{4}.$$
In this case $G_R$ is not a strongly regular graph (since they are connected with more than 3 eigenvalues). 
\end{rem}

\begin{exam}
Suppose $R=\ff_q \times \ff_q \times \ff_q \times \ff_q \times \ff_q$. By \eqref{ff prod} we have 
$$5(q-1)^2+5(q-1)^4 = (q-1)^5$$ or, equivalently, $5((q-1)^2+1) = (q-1)^3$. 
Since the equation $x^3-5x^2-5$ has only one real root which is not an integer, we obtain that there are no unitary Cayley graphs $G_R$ complementary equienergetic, with $R$ a product of 5 copies of a finite field.  
\end{exam}

The determination of all complementary equienergetic unitary Cayley graphs $G_R$ in the general case with an odd number $s$ of local factors remains as on open problem. The difficulty in this particular open case, shows the complexity of the problem of classifying all complementary equienergetic non-bipartite regular graphs (the open case $C(1)$ in Cameron's hierarchy) which are not strongly regular. In this respect, we can only mention the pairs $L(K_{m,n})$ and $\overline{L(K_{m,n})}$ of Ramane et al (\cite{RPPA}), with $m,n\ge 2$ and $m\ne n$ (for $m=n$ the graphs are strongly regular since $L({K_{n,n}})=L_2(n)$).

\subsubsection*{Unitary Cayley sum graphs}
A variant of the unitary Cayley graphs $G_R=X(R,R^*)$, where $R$ is a commutative Artinian ring with identity,  is the unitary Cayley sum graph $G_R^+=X^+(G,S)$ where two vertices $v,w$ form and edge if and only if $v+w\in R^*$. These graphs are $|R^*|$-regular and non-directed, but may contain single loops (there is a loop in $x$ if and only if $2x \in R^*$). If $\rm{char}(R)=2$ the graph $G_R^+=G_R$, hence loopless. If $char(R)$ is odd then $G_R^+$ has $|R^*|$ single loops. 

\begin{prop} \label{GR+}
	There are no complementary equienergetic unitary Cayley sum graphs $G_R^+$ with loops.
\end{prop}

\begin{proof}
By Lemma 3.4 in \cite{PV5}, if $R$ is not of odd type (see Definition 3.5 in \cite{PV5}) then $G_R^+=G_R$ and hence it is loopless. So, we assume that $R$ is of odd type. By Theorem 3.7 in \cite{PV5}, $G_R^+$ is an integral connected non-bipartite graph. Moreover, if $R$ has odd cardinality $r$ then $G_R^+$ is strongly almost symmetric with $r^*$ single loops (where $r^*=|R^*|$), while if $r$ is even $G_R^+$ is loopless. So, we assume also that $r$ is odd. By Corollary \ref{comp equi loops}, $G_R^+$ is complementary equienergetic if and only if $n=2k$, that is $r=2r^*$ in our present notation. But this is impossible since $r$ is odd.  
\end{proof}

It should be interesting to study the case of the graphs $\G=(G_R^+)^{**}$ which are regular with at most 2 loops per vertex.
For instance, if $R$ is local (with $r$ odd), by Proposition~\ref{lem delta} with $m=2$, $\G$ is complementary equienergetic if and only if $n=2k-1-\Delta_1(\G)$, where $k=r^*+1$. But if $\mathfrak{m}$ is the unique maximal ideal of $R$ with $t$ elements, then $r^*=r-t$ and thus the required condition reads
$$r+1=2t-\Delta_1(\G).$$ 
Hence, a necessary condition is that $\Delta_1(\G) \in 2\N$. Since $\G$ is integral, this means that there must be an even number of positive eigenvalues different from $\lambda_1$.

\section{Final remarks}
We have given some complete characterizations of complementary equienergetic regular graphs for some families of graphs. Namely, connected integral cubic graphs and cubic graphs with single loops (\S 3), bipartite graphs (\S 4), strongly regular graphs in general (\S 5) and some particular families (\S 6), graphs having OA parameters (\S 7) and unitary Cayley graphs over rings $R$ with $R$ a commutative ring with unity which is either local or with an even number of local factors (\S 8). 

We want to notice that all complementary equienergetic regular graphs (not isospectral with their complements) covered in the results in this paper, as well as those mentioned in the introduction (families and examples of complementary equienergetic regular graphs in \cite{A+}, \cite{RPPAG}, \cite{RGWH}, \cite{RPPA}), have integral energy divisible by 4. The only exception we could find are the pairs of graphs of $10$ vertices $\{H_{55}, H_{56}\}$ and $\{H_{57}, H_{58}\}$ in \cite{A+}. They have non-integral energy. However, they are regular but neither cubic, nor bipartite nor strongly regular. 

It would be interesting to study this phenomenon in more detail. Which families of regular graphs complementary equienergetic and non-isospectral with their complements have integral energy divisible by 4 and why? In particular, it would be nice to characterize complementary equienergetic graphs in some other general families like: ($a$) distance-transitive or distance-regular graphs, ($b$) $k$-iterated line graphs, ($c$) Deza graphs and ($d$) some families of Cayley graphs. We plan to study these problems in the future. The general case of non-bipartite (regular or not) or non-regular graphs seems rather hard and out of scope.

\end{document}